\documentclass[11pt, american, a4paper]{amsart}

\usepackage{times}

\usepackage{amssymb}
\usepackage{babel}
\usepackage{enumitem}
\usepackage{hyperref}
\usepackage[utf8]{inputenc}
\usepackage{newunicodechar}
\usepackage{svn-multi}
\usepackage{varioref}
\usepackage[arrow,curve,matrix]{xy}

\usepackage{colortbl}
\usepackage{graphicx}
\usepackage{tikz}

\usepackage{mathrsfs}

\numberwithin{figure}{section}

\usepackage[hyperpageref]{backref}

\setdescription{labelindent=\parindent, leftmargin=2\parindent}
\setitemize[1]{labelindent=\parindent, leftmargin=2\parindent}
\setenumerate[1]{labelindent=0cm, leftmargin=*, widest=iiii}

\newunicodechar{α}{\ensuremath{\alpha}}
\newunicodechar{β}{\ensuremath{\beta}}
\newunicodechar{χ}{\ensuremath{\chi}}
\newunicodechar{δ}{\ensuremath{\delta}}
\newunicodechar{ε}{\ensuremath{\varepsilon}}
\newunicodechar{Δ}{\ensuremath{\Delta}}
\newunicodechar{η}{\ensuremath{\eta}}
\newunicodechar{γ}{\ensuremath{\gamma}}
\newunicodechar{Γ}{\ensuremath{\Gamma}}
\newunicodechar{ι}{\ensuremath{\iota}}
\newunicodechar{κ}{\ensuremath{\kappa}}
\newunicodechar{λ}{\ensuremath{\lambda}}
\newunicodechar{Λ}{\ensuremath{\Lambda}}
\newunicodechar{μ}{\ensuremath{\mu}}
\newunicodechar{ω}{\ensuremath{\omega}}
\newunicodechar{Ω}{\ensuremath{\Omega}}
\newunicodechar{π}{\ensuremath{\pi}}
\newunicodechar{φ}{\ensuremath{\phi}}
\newunicodechar{Φ}{\ensuremath{\Phi}}
\newunicodechar{ψ}{\ensuremath{\psi}}
\newunicodechar{Ψ}{\ensuremath{\Psi}}
\newunicodechar{ρ}{\ensuremath{\rho}}
\newunicodechar{σ}{\ensuremath{\sigma}}
\newunicodechar{Σ}{\ensuremath{\Sigma}}
\newunicodechar{τ}{\ensuremath{\tau}}
\newunicodechar{θ}{\ensuremath{\theta}}
\newunicodechar{Θ}{\ensuremath{\Theta}}

\newunicodechar{→}{\ensuremath{\to}}
\newunicodechar{⨯}{\ensuremath{\times}}
\newunicodechar{∪}{\ensuremath{\cup}}
\newunicodechar{∩}{\ensuremath{\cap}}
\newunicodechar{⊇}{\ensuremath{\supseteq}}
\newunicodechar{⊃}{\ensuremath{\supset}}
\newunicodechar{⊆}{\ensuremath{\subseteq}}
\newunicodechar{⊂}{\ensuremath{\subset}}
\newunicodechar{≥}{\ensuremath{\geq}}
\newunicodechar{≤}{\ensuremath{\leq}}
\newunicodechar{∈}{\ensuremath{\in}}
\newunicodechar{◦}{\ensuremath{\circ}}
\newunicodechar{°}{\ensuremath{^\circ}}
\newunicodechar{…}{\ifmmode\mathellipsis\else\textellipsis\fi}

\newunicodechar{¹}{\ensuremath{^1}}
\newunicodechar{²}{\ensuremath{^2}}
\newunicodechar{³}{\ensuremath{^³}}

\DeclareFontFamily{OMS}{rsfs}{\skewchar\font'60}
\DeclareFontShape{OMS}{rsfs}{m}{n}{<-5>rsfs5 <5-7>rsfs7 <7->rsfs10 }{}
\DeclareSymbolFont{rsfs}{OMS}{rsfs}{m}{n}
\DeclareSymbolFontAlphabet{\scr}{rsfs}
\DeclareSymbolFontAlphabet{\scr}{rsfs}

\DeclareMathOperator{\Aut}{Aut}
\DeclareMathOperator{\codim}{codim}
\DeclareMathOperator{\coker}{coker}
\DeclareMathOperator{\Ext}{Ext}

\DeclareMathOperator{\Id}{Id}

\DeclareMathOperator{\Pic}{Pic}
\DeclareMathOperator{\rank}{rank}

\DeclareMathOperator{\reg}{reg}

\DeclareMathOperator{\Spec}{Spec}
\DeclareMathOperator{\Sym}{Sym}

\newcommand{\sA}{\scr{A}}
\newcommand{\sB}{\scr{B}}
\newcommand{\sC}{\scr{C}}

\newcommand{\sE}{\scr{E}}
\newcommand{\sF}{\scr{F}}

\newcommand{\sI}{\scr{I}}

\newcommand{\sL}{\scr{L}}

\newcommand{\sN}{\scr{N}}
\newcommand{\sO}{\scr{O}}

\newcommand{\bA}{\mathbb{A}}

\newcommand{\bC}{\mathbb{C}}

\newcommand{\bF}{\mathbb{F}}
\newcommand{\bG}{\mathbb{G}}

\newcommand{\bN}{\mathbb{N}}

\newcommand{\bP}{\mathbb{P}}

\newcommand{\bR}{\mathbb{R}}

\newcommand{\bZ}{\mathbb{Z}}

\theoremstyle{plain}
\newtheorem{thm}{Theorem}[section]

\newtheorem{cor}[thm]{Corollary}
\newtheorem{defn}[thm]{Definition}
\newtheorem{fact}[thm]{Fact}
\newtheorem{lem}[thm]{Lemma}

\newtheorem{prop}[thm]{Proposition}

\theoremstyle{remark}
\newtheorem{assumption}[thm]{Assumption}
\newtheorem{asswlog}[thm]{Assumption w.l.o.g.}
\newtheorem{claim}[thm]{Claim}
\newtheorem{c-n-d}[thm]{Claim and Definition}

\newtheorem{construction}[thm]{Construction}
\newtheorem{computation}[thm]{Computation}

\newtheorem{explanation}[thm]{Explanation}
\newtheorem{notation}[thm]{Notation}
\newtheorem{obs}[thm]{Observation}
\newtheorem{rem}[thm]{Remark}
\newtheorem{question}[thm]{Question}
\newtheorem*{rem-nonumber}{Remark}
\newtheorem{setting}[thm]{Setting}

\numberwithin{equation}{thm}

\setlist[enumerate]{label=(\thethm.\arabic*), before={\setcounter{enumi}{\value{equation}}}, after={\setcounter{equation}{\value{enumi}}}}

\def\clap#1{\hbox to 0pt{\hss#1\hss}}

\newcommand\CounterStep{\addtocounter{thm}{1}\setcounter{equation}{0}}

\newcommand{\factor}[2]{\left. \raise 2pt\hbox{$#1$} \right/\hskip -2pt\raise -2pt\hbox{$#2$}}%
\definecolor{linkred}{rgb}{0.7,0.2,0.2}
\definecolor{linkblue}{rgb}{0,0.2,0.6}

\newcommand{\Publication}[1]{}

\DeclareFontFamily{OT1}{pzc}{}
\DeclareFontShape{OT1}{pzc}{m}{it}{<-> s * [1.195] pzcmi7t}{}
\DeclareMathAlphabet{\mathscr}{OT1}{pzc}{m}{it}

\usepackage[top=4cm, bottom=3cm, left=3.5cm, right=3.5cm]{geometry}

\newcommand{\tensor}{\otimes}
\newcommand{\isomto}{{\stackrel{\sim}{\;\longrightarrow\;}}}
\newcommand{\isomt}{{\stackrel{{\scriptscriptstyle{\sim}}}{\;\rightarrow\;}}}
\newcommand{\gm}{{{\mathbf G}_{m}}}
\newcommand{\et}{\text{ét}}
\newcommand{\ho}[1]{\mathscr{H}({#1})}
\newcommand{\bpi}{\boldsymbol{π}}
\newcommand{\Nis}{\operatorname{Nis}}
\newcommand{\Zar}{\operatorname{Zar}} 
\newcommand{\Sm}{\mathscr{Sm}}
\newcommand{\Spc}{\mathscr{Spc}}
\newcommand{\K}{{{\mathbf K}}}
\newcommand{\Grass}{\ensuremath{\mathscr{Gr}_n(k)}}
\newcommand{\aoequiv}{\ensuremath{\simeq_{\bA^1}}}
\newcommand{\hcob}[1]{\ensuremath{\mathscr{#1}}}
\DeclareMathOperator{\Bs}{Bs}
\DeclareMathOperator{\sProj}{\bf Proj}
\DeclareMathOperator{\Proj}{Proj}

\theoremstyle{remark}
\newtheorem{conventions}[thm]{Conventions}

\title{Comparing $\bA^1$-$h$-cobordism and $\bA^1$-weak equivalence}
\date{\today}

\author{Aravind Asok} %
\address{Aravind Asok, Department of Mathematics, University of Southern
  California, 3620 S. Vermont Ave KAP 104, Los Angeles CA, 90089-2532} %
\email{\href{mailto:asok@usc.edu}{asok@usc.edu}}
\urladdr{\href{http://www-bcf.usc.edu/~asok}{http://www-bcf.usc.edu/$\sim$asok}}

\author{Stefan Kebekus} %
\address{Stefan Kebekus, Mathematisches Institut, Albert-Ludwigs-Universität,
  Eckerstr.  1, 79104 Freiburg and University of Strasbourg Institute for
  Advanced Study (USIAS), Strasbourg, France}
\email{\href{mailto:stefan.kebekus@math.uni-freiburg.de}{stefan.kebekus@math.uni-freiburg.de}}
\urladdr{\href{http://home.mathematik.uni-freiburg.de/kebekus}{http://home.mathematik.uni-freiburg.de/kebekus}}

\author{Matthias Wendt} %
\address{Matthias Wendt, Fakultät für Mathematik, Universität Duisburg-Essen,
  Thea-Leymann-Strasse 9, 45127 Essen}
\email{\href{mailto:matthias.wendt@uni-due.de}{matthias.wendt@uni-due.de}}

\thanks{Aravind Asok was partially supported by National Science Foundation
  Awards DMS-1254892.  Stefan Kebekus was supported in part by the
  DFG-Forschergruppe 790.  He also acknowledges support through a joint
  fellowship of the Freiburg Institute of Advanced Studies (FRIAS) and the
  University of Strasbourg Institute for Advanced Study (USIAS). Matthias Wendt
  was partially supported by the Alexander-von-Humboldt-Stiftung and by the
  DFG-Sonderforschungsbereich SFB/TR~45.}



\definecolor{linkred}{rgb}{0.7,0.2,0.2}
\definecolor{linkblue}{rgb}{0,0.2,0.6}
\makeatletter
\hypersetup{
  pdfauthor={\authors},
  pdftitle={\@title},
  pdfsubject={\@subjclass},
  pdfkeywords={\@keywords},
  pdfstartview={Fit},
  pdfpagelayout={TwoColumnRight},
  pdfpagemode={UseOutlines},
  bookmarks,
  colorlinks,
  linkcolor=linkblue,
  citecolor=linkred,
  urlcolor=linkred
}
\makeatother

\begin{document}

\maketitle

%
%

\begin{abstract}
  We study the problem of classifying projectivizations of rank-two vector
  bundles over $\bP²$ up to various notions of equivalence that arise naturally
  in $\bA¹$-homotopy theory, namely $\bA¹$-weak equivalence and
  $\bA¹$-$h$-cobordism.

  First, we classify such varieties up to $\bA¹$-weak equivalence: over
  algebraically closed fields having characteristic unequal to two the
  classification can be given in terms of characteristic classes of the
  underlying vector bundle.  When the base field is $\bC$, this classification
  result can be compared to a corresponding topological result and we find that
  the algebraic and topological homotopy classifications agree.

  Second, we study the problem of classifying such varieties up to
  $\bA¹$-$h$-cobordism using techniques of deformation theory.  To this end, we
  establish a deformation rigidity result for $\bP¹$-bundles over $\bP²$ which
  links $\bA¹$-$h$-cobordisms to deformations of the underlying vector bundles.
  Using results from the deformation theory of vector bundles we show that if
  $X$ is a $\bP¹$-bundle over $\bP²$ and $Y$ is the projectivization of a direct
  sum of line bundles on $\bP²$, then if $X$ is $\bA¹$-weakly equivalent to $Y$,
  $X$ is also $\bA¹$-$h$-cobordant to $Y$.

  Finally, we discuss some subtleties inherent in the definition of
  $\bA¹$-$h$-cobordism.  We show, for instance, that direct $\bA¹$-$h$-cobordism
  fails to be an equivalence relation.
\end{abstract}

\setcounter{tocdepth}{1}
\tableofcontents

%
%

\section{Introduction}

In this note, we study the relation of two classification problems in the
topology of algebraic varieties.  On the one hand, there is the problem of
classifying smooth proper varieties over a fixed field $k$ up to $\bA¹$-weak
equivalence.  We refer to this as the $\bA¹$-homotopy classification problem.
On the other hand, there is the problem of classifying smooth proper varieties
having a fixed $\bA¹$-homotopy type.  This is an analogue of the surgery
problem in differential topology.  These problems were initially studied in
\cite{AM} for varieties of dimension at most two.  For this the notion of
$\bA¹$-$h$-cobordism of smooth proper varieties was introduced as an algebraic
replacement of $h$-cobordism of smooth manifolds.  By definition, varieties that
are $\bA¹$-$h$-cobordant are $\bA¹$-weakly equivalent and in \cite{AM} examples
are produced to show that $\bA¹$-$h$-cobordant varieties need not be isomorphic.

The present work takes the next step, studying these classification problems in
dimension three.  The varieties we consider are projectivizations of rank-two
vector bundles on the projective plane $\bP²$ over a fixed base field, which
will be suppressed from the notation.

\subsection{Classification up to $\bA¹$-weak equivalence}

As a first result in this direction, we can provide a complete classification of
such varieties up to $\bA¹$-weak equivalence, at least for certain
base fields.

\begin{thm}[= Theorem~\vref{thm:a1homotopyp1bundles}]\label{thmintro:main1}
  Assume $k$ is an algebraically closed field having characteristic unequal to
  two.  If $\sE$ and $\sF$ are two vector bundles over the projective plane over
  $\bP²$, each of rank two, then the following are equivalent.
  \begin{enumerate}
  \item The pairs of Chern classes $\bigl( c_1(\sE), c_2(\sE) \bigr)$ and
    $\bigl( c_1(\sF), c_2(\sF) \bigr)$ are in the same orbit for the action of
    $\Pic(\bP²)$ on $\Pic(\bP²)⨯ CH²(\bP²)$ induced from twisting by
    line bundles, cf.\ Theorem~\ref{thm:pgl2torsors} and
    Corollary~\ref{cor:pgl2torsorsonp2}.
  \item There is an $\bA¹$-weak equivalence $\bP_{\bP²}(\sE) \aoequiv
    \bP_{\bP²}(\sF)$.
  \end{enumerate}
\end{thm}

To establish this result, we first provide an $\bA¹$-homotopy classification of
$PGL_2$-torsors over $\bP²$.  This classification is obtained by appeal to
techniques of obstruction theory, cf.\ Theorem~\ref{thm:pgl2torsors} and
Corollary~\ref{cor:pgl2torsorsonp2}.  Results from the theory of fiber sequences
then show that homotopies of classifying maps of Zariski locally trivial
$\bP¹$-bundles yield $\bA¹$-weak equivalences of total spaces, cf.\
Corollary~\ref{cor:weaklyequivalentbundleshaveweaklyequivalenttotalspaces}.
Conversely, using the cubic form on the Picard group and some results from
classical invariant theory, we show that if two $\bP¹$-bundles over $\bP²$ are
$\bA¹$-weakly equivalent, then the underlying vector bundles have the same
classifying maps.  The weak equivalence between total spaces is then induced from
a homotopy between those classifying maps; roughly speaking ``every $\bA¹$-weak
equivalence between the total spaces of Zariski locally trivial $\bP¹$-bundles
is induced by a fiber homotopy equivalence''.

\subsection{Classification up to cobordism}

The second part of the paper is devoted to understanding $\bA¹$-$h$-cobordism
classes within a given $\bA¹$-homotopy type in the special case of
$\bP¹$-bundles over $\bP²$.  We obtain the following partial classification
result, which exhibits some interesting subtleties of the notion of
$\bA¹$-$h$-cobordism.

\begin{thm}[= Proposition~\vref{prop:cobordismclass} and Theorem~\vref{thm:a1cobx}]\label{thmintro:main2}
  Let $k$ be an algebraically closed field having characteristic unequal to two,
  and let $c_1, c_2 ∈ \bZ$ be integers.  The following results concerning
  $\bA¹$-$h$-cobordism classes of rank-two vector bundles on $\bP²$ with Chern
  classes $c_1$ and $c_2$ hold.
  \begin{enumerate}
  \item\label{il:bonny} If there exists an integer $d$ such that
    $d²-dc_1+c_2=0$, then for any two rank-two vector bundles $\sE$ and $\sF$ on
    $\bP²$ with Chern classes $c_1$ and $c_2$, the corresponding projective
    bundles $\bP_{\bP²}(\sE)$ and $\bP_{\bP²}(\sF)$ are $\bA¹$-$h$-cobordant.
    In particular, any $\bP¹$-bundle over $\bP²$ which is $\bA¹$-weakly
    equivalent to $\bP¹ ⨯ \bP²$ is also $\bA¹$-$h$-cobordant to $\bP¹ ⨯ \bP²$.

  \item\label{il:clyde} There are infinitely many rank-two vector bundles
    $\bigl( \sE_i \bigr)_{i ∈ \bN}$ on $\bP²$ with Chern classes $c_1$ and $c_2$
    such that no two of the varieties $\bP_{\bP²}(\sE_i)$ are directly
    $\bA¹$-$h$-cobordant.  In particular, direct $\bA¹$-$h$-cobordism is rather
    far from being an equivalence relation.
  \end{enumerate}
\end{thm}

These results rely on a certain deformation rigidity result, which provides a
close relation between $\bA¹$-$h$-cobordisms and $\bA¹$-deformations of the
underlying vector bundles: given an $\bA¹$-$h$-cobordism with $f^{-1}(0)$ a
given $\bP¹$-bundle over $\bP²$, there is a Zariski open neighborhood of $0 ∈
\bA¹$ over which the resulting deformation is induced from a deformation of
rank-two vector bundles over $\bP²$; see Theorem~\vref{thm:rigidity} for a
precise statement.  This result allows us to import some results of Strømme,
\cite{stroemme}, to help investigate the $\bA¹$-$h$-cobordism classification of
$\bP¹$-bundles over $\bP²$.  Assuming some condition on the Chern classes, like
those imposed in Part~\ref{il:bonny} of Theorem~\ref{thmintro:main2}, we observe
that there are ``enough" deformations of rank-two vector bundles over $\bP²$ to
guarantee that the $\bA¹$-$h$-cobordism classification is not finer than the
$\bA¹$-homotopy classification, cf.  Proposition~\ref{prop:cobordismclass}.  On
the other hand, the non-deformability results of Strømme imply the existence of
infinitely many varieties in each of the above $\bA¹$-weak homotopy types which
cannot be connected by direct $\bA¹$-$h$-cobordisms.  These observations lead to
the results spelled out in Part~\ref{il:clyde} of Theorem~\ref{thmintro:main2}.

Our $\bA¹$-$h$-cobordism classification result is incomplete because we impose
restrictions on the Chern classes of the vector bundles under consideration.
The main reason for these restrictions stems from the difficulties inherent in
providing an isomorphism or deformation classification of vector bundles on
$\bP²$.  While the explicit families of vector bundles produced in Strømme's
work are enough to prove connectedness of the ``moduli space of rank two bundles
on $\bP²$'', they do not allow us to establish the $\bA¹$-chain connectedness of
that space.  At the moment, we are unable to decide if Part~\ref{il:bonny} of
Theorem~\ref{thmintro:main2} can be extended to all projective bundles or if
there exist projective bundles which are $\bA¹$-weakly equivalent but not
$\bA¹$-$h$-cobordant.

Finally, we take a moment to indicate more abstractly the main difficulties
involved in the study of $\bA¹$-$h$-cobordism.  Our problem, phrased a bit more
broadly, is to understand all the smooth proper varieties having a fixed
$\bA¹$-homotopy type, say modulo various notions of equivalence.  Varieties in a
fixed $\bA¹$-homotopy type can appear in families.  Thus, it is natural to try
to construct a ``moduli space of scheme structures in a fixed $\bA¹$-homotopy
type." In order to analyze $\bA¹$-$h$-cobordisms, we would ideally like this
moduli problem to be representable by a smooth scheme: if that was true, then we
could try to construct $\bA¹$-$h$-cobordisms by producing maps from $\bA¹$ to
the moduli space.  However, difficulties arise involving both of these ideas.
Indeed, the moduli problem need not be representable by a smooth scheme, and it
turns out to be hard to construct $\bA¹$-$h$-cobordisms.

Already in the case of $\bP¹$-bundles over $\bP²$ the moduli problem is not
representable by a smooth scheme.  Nevertheless, after fixing an additional
invariant, the generic splitting type, one can construct suitable moduli schemes
within the $\bA¹$-homotopy type, though we make no claim that these moduli
schemes actually exhaust the $\bA¹$-homotopy type.  Indeed, it seems likely that
there are smooth projective varieties that are $\bA¹$-$h$-cobordant to
$\bP¹$-bundles over $\bP²$ but that are not themselves of this form.  In this
case, deformations of a vector bundle parameterized by the affine line give rise
to $\bA¹$-$h$-cobordisms, so there is a close connection between affine lines on
the moduli space and $\bA¹$-$h$-cobordisms.  The results of Strømme establish
that the moduli problem as a whole is connected (in the usual topology) for each
$\bA¹$-homotopy type of $\bP¹$-bundles over $\bP²$, but has infinitely many
irreducible components.  These observations lead to the failure of
$\bA¹$-$h$-cobordism to be an equivalence relation.

We summarize these observations as a slogan: $\bA¹$-$h$-cobordism is sensitive
to the geometry of moduli of scheme structures.  In fact, it seems likely that
$\bA¹$-$h$-cobordism is only well-behaved if the ``moduli space of scheme
structures'' is well-behaved, say, locally $\bA¹$-contractible.  In view of
Murphy's law for moduli spaces, \cite{Vakil}, this ``local
$\bA¹$-contractibility of the moduli space of scheme structures'' is likely to
hold, if ever, only in very special cases.

\subsection{Structure of the paper}

Section~\ref{sec:02} gathers a number of results concerning rank-two vector
bundles on $\bP²$ and their associated projective bundles.  To the best of our
knowledge, some of these results are new and might be of independent interest.
In Section~\ref{sec:03} we discuss the $\bA¹$-homotopy classification of
$PGL_2$-torsors over $\bP²$, from which we deduce the $\bA¹$-homotopy
classification of $\bP¹$-bundles over $\bP²$ in Section~\ref{sec:04}.  Then, we
turn to the more geometric equivalence relations.  We define $\bA¹$-concordance
and discuss the classification of rank-two bundles over $\bP²$ up to
$\bA¹$-concordance in Section~\ref{sec:05}; consequences of these results for
the $\bA¹$-$h$-cobordism classification of $\bP¹$-bundles over $\bP²$ are
contained in Section~\ref{sec:06}.  Finally, in Section~\ref{sec:07}, we compare
the algebraic classification results with corresponding topological
classification results in the setting of complex manifolds.

\subsection{Notation and global conventions}
\label{sec:conventions}

Throughout this paper, we work with schemes that are separated and have finite
type over an algebraically closed field $k$.  With the exception of
Section~\ref{sec:03}, the characteristic of $k$ will always be unequal to two.
Following notation from Hartshorne's book, an \emph{abstract variety} is an
integral, separated scheme of finite type over $k$.  We use the word ``sheaf''
to mean ``coherent sheaf'', unless noted otherwise.

Throughout this paper, we fix a hyperplane class $H$ on $\bP²$ and use this to
identify $\Pic(\bP²) \cong \bZ \cdot H$ and $CH²(\bP²) \cong \bZ \cdot H²$.  If
$\sE$ is a rank-two vector bundle on $\bP²$, we use these identifications to
view the Chern classes $c_i(\sE)$ as integers.

%
%

\section{Vector bundles over \texorpdfstring{$\bP²$}{P\texttwosuperior} and associated projective bundles}
\label{sec:02}

For the reader's convenience, we briefly recall in this section notation and
results pertaining to vector bundles on $\bP²$, to families of vector bundles,
and to their moduli spaces.  Section~\ref{subsec:hscorrespondence} begins by
recalling some results about the Hartshorne-Serre correspondence relating vector
bundles to codimension-two local complete intersections.  We are particularly
interested in the relative setting.  Section~\ref{subsec:bertini} recalls a
Bertini-type theorem, which appears in the work of Kleiman.
Section~\ref{sec:2B} contains a uniqueness result about projective bundle
structures, see Theorem~\ref{thm:uniq}.  Section~\ref{sec:2D} contains the
deformation rigidity result mentioned in the introduction,
Theorem~\ref{thm:rigidity}.  Finally, Section~\ref{ssec:deformation} recalls
some results about the deformation theory of vector bundles on $\bP²$.  With the
exception of Section~\ref{ssec:deformation}, which requires the notion of
\emph{type} of a vector bundle, see Definition~\ref{def:type}, these sections
are written to be independent of each other.

\subsection{The Hartshorne-Serre correspondence}
\label{subsec:hscorrespondence}

We will briefly recall the well-known correspondence between rank-two vector
bundles on a given smooth variety $X$ and codimension-two local complete
intersections $Y ⊂ X$.  The following simplified version suffices for our
purposes.

\begin{fact}[\protect{Hartshorne-Serre correspondence, \cite[Thm.~1.1]{Arrondo}}]\label{fact:arrondo1}
  Let $X$ be any smooth variety of dimension $\dim X ≥ 2$, and $Y ⊆ X$ be a
  local complete intersection of codimension two, with ideal sheaf $\sI_Y ⊂
  \sO_X$.  Let $\sL ∈ \Pic(X)$ be any line bundle such that the dualizing sheaf
  $ω_Y$ is isomorphic to $(\sL \otimes ω_X)|_Y$.  Then, there exists a
  canonically defined, functorial exact sequence
  \begin{multline}\label{eq:gfk}
    H¹ \bigl( X,\, \sL^* \bigr) → \Ext¹ \bigl( \sI_Y \otimes \sL,\, \sO_X \bigr) \\
    → H^0 \bigl( Y,\, \wedge² \sN_{X/Y}\otimes\sL^\ast|_Y \bigr) → H² \bigl(
    X,\, \sL^* \bigr).  \qed
  \end{multline}
\end{fact}

\begin{rem}\label{ram:arrondo2}
  In the setting of Fact~\ref{fact:arrondo1}, if we assume in addition that the
  cohomology groups $H¹ \bigl( X,\, \sL^* \bigr)$ and $H² \bigl( X,\, \sL^*
  \bigr)$ vanish, then Sequence~\eqref{eq:gfk} yields a canonical isomorphism
  $\Ext¹ \bigl( \sI_Y \otimes \sL,\, \sO_X \bigr) \cong H^0 \bigl( Y,\, \wedge²
  \sN_{X/Y}\otimes\sL^\ast|_Y \bigr)$.
\end{rem}

\begin{thm}[Characterisation of locally frees, I]\label{thm:arrondo2}
  In the setting of Fact~\ref{fact:arrondo1}, given an extension of the form
  \begin{equation}\label{eq:slpt}
    \xymatrix{ %
      0 \ar[r] & \sO_X \ar[r]^s & \sE \ar[r] & \sI_Y \tensor \sL \ar[r] & 0,
    }
  \end{equation}
  then $\sE$ is locally free if any only if the section of $\wedge²
  \sN_{X/Y}\otimes\sL^\ast$ that is associated with the extension class of
  \eqref{eq:slpt} generates that sheaf.
\end{thm}
\begin{proof}
  This result is established in \cite[Thm.~1.1]{Arrondo} under the additional
  assumption that the cohomology groups $H¹ \bigl( X,\, \sL^* \bigr)$ and $H²
  \bigl( X,\, \sL^* \bigr)$ vanish.  Note that Sequence~\eqref{eq:gfk} is
  functorial with respect to restriction maps.  By picking an open affine cover
  of $X$, we can always guarantee the necessary cohomology vanishing, and local
  freeness can be checked by restriction to each open affine in the cover.
\end{proof}

\begin{rem}[Characterisation of locally frees, II]
  Let $X$ be any smooth variety of dimension $\dim X ≥ 2$, and $Y ⊆ X$ be a
  local complete intersection of codimension two, with ideal sheaf $\sI_Y ⊂
  \sO_X$.  Given an extension of the form \eqref{eq:slpt}, observe that the
  sheaf $\sE$ is locally free if and only if the section $s$ vanishes precisely
  on $Y$.
\end{rem}

\begin{rem}[Hartshorne-Serre correspondence for bundles on $\bP²$]\label{rem:specialcasep2}
  Consider the case where $X = \bP²$, where $Y ⊂ X$ is any finite, reduced
  subscheme and where $\sL \cong \sO_{\bP²}(d)$ with $d < 3$.  It will then
  follow directly from Serre duality that $H² \bigl( X,\, \sL^* \bigr) = 0$.
  The assumption that $ω_Y$ be isomorphic to $(\sL \otimes ω_X)|_Y$ is vacuous
  in this case.  The bundle $\bigwedge² \sN_{X/Y}\otimes\sL^\ast|_Y$ is the
  trivial line bundle on $Y$.  Since $H¹ \bigl(\bP²,\, \sL^* \bigr) = 0$, each
  nowhere-vanishing section $σ$ in $H^0 \bigl(Y,\, \sO_Y \bigr)$ gives rise to a
  (unique up to isomorphism) rank-two vector bundle on $\bP²$.
\end{rem}

The following corollary applies this result.  It will later be used to construct
deformations of the bundle $\sE$ by moving points in $Y$ within $\bP²$.  The
following notation will be useful.

\begin{notation}
  Using that $\Pic(\bP²) \cong \Pic(\bP²⨯\bA¹)$, identify $\Pic(\bP²⨯\bA¹) \cong
  \bZ$.  Given any integer $n$, write $\sO_{\bP²⨯\bA¹}(n)$ for the corresponding
  line bundle.  In a similar vein, identify $CH²(\bP²⨯\bA¹) \cong \bZ$.  Given
  any rank-two vector bundle $\sE$ on $\bP²⨯\bA¹$, we can thus identify the
  Chern classes $c_i(\sE)$ with integers.
\end{notation}

\begin{cor}[Extension of vector bundles]\label{cor:rhsc}
  Consider the quasi-projective variety $X := \bP²⨯\bA¹$ and the projection onto
  the second factor $π: X → \bA¹$.  Let $Y ⊂ X$ be the union of $m$ pairwise
  disjoint sections of $π$ and $d ≤ 2$ be any integer.  Write $X_0 :=
  \bP²⨯\{0\}$ and $Y_0 := X_0 ∩ Y$.  Assume we are given a rank-two bundle
  $\sE_0$ on $X_0$, defined by an extension,
  \begin{equation}\label{eq:sdf4}
    \xymatrix{ %
      0 \ar[r] & \sO_{X_0} \ar[r]^{s_0} & \sE_0 \ar[r] & \sI_{Y_0} \tensor
      \sO_{X_0}(d) \ar[r] & 0.  }
  \end{equation}
  Then, there exists a rank-two bundle $\sE$ on $X$, defined by an extension
  $$
  \xymatrix{ %
    0 \ar[r] & \sO_X \ar[r]^s & \sE \ar[r] & \sI_Y \otimes \sO_X(d) \ar[r] & 0,
  }
  $$
  such that $\sE|_{X_0} \cong \sE_0$ and $s|_{X_0} = s_0$.
\end{cor}

\subsection*{Proof of Corollary~\ref*{cor:rhsc}}

For the reader's convenience, the proof is subdivided into several steps.

\subsubsection*{Step 1: Establishing prerequisites for the Hartshorne-Serre correspondence}

In order to construct the bundle $\sE$, we aim to apply the results of
Fact~\ref{fact:arrondo1} and Theorem~\ref{thm:arrondo2} to $X$, with $\sL =
\sO_X(d)$.  Observe that $Y$ is a local complete intersection, being a smooth,
closed subscheme of $X$.  Since $Y$ is isomorphic to a disjoint union of $m$
copies of $\bA¹$, it follows that all locally free sheaves on $Y$ are free.  The
assumption that $ω_Y$ be isomorphic to $(\sL \otimes ω_X)|_Y$ is therefore
vacuous.

In order to verify vanishing of $H¹ \bigl( X,\, \sL^* \bigr)$ and $H² \bigl(
X,\, \sL^* \bigr)$, consider the Leray spectral sequence associated with $π$,
\cite[Chapt.~II.4.17]{MR0345092}, which takes the form
$$
E²_{ij} = H^i \bigl( \bA¹,\, \bR^j π_* \sL^* \bigr) \Longrightarrow H^{i+j}
\bigl( X,\, \sL^* \bigr).
$$
Since $\bA¹$ is affine, the cohomology groups $H^i \bigl( \bA¹,\, \sF \bigr)$
vanish for any quasi-coherent sheaf $\sF$ and any number $i > 0$.  In
particular, the spectral sequence collapses at the $E²$-page, \cite[Chapt.~I.1,
Ex.~1.B]{MR1793722}, and yields isomorphisms
\begin{equation}\label{eq:A}
  H^0 \bigl( \bA¹,\, \bR^j π_* \sL^* \bigr) \cong
  H^j \bigl(X,\, \sL^* \bigr) \quad \text{for all $j ≥ 0$.}
\end{equation}
Identify $X=\bP²⨯\bA¹$ with the projectivization of the trivial rank-three
bundle on $\bA¹$.  With this identification, it follows from the special case of
relative duality discussed in \cite[Chapt.~III, Exc.~8.4c]{Ha77} that there is a
canonical isomorphism
\begin{equation}\label{eq:B}
  \bR² π_* \sL^* \cong \left( π_* (\sL \otimes ω_X) \right)^* \cong
  \left( π_* \sO_X(d-3) \right)^* = 0 \quad \text{since } d≤2.
\end{equation}
Combining \eqref{eq:A} and \eqref{eq:B}, we see that $H² \bigl( X,\, \sL^*
\bigr) = 0$.  A somewhat simpler argument, left to the reader, shows that
$H¹\bigl( X,\, \sL^* \bigr)$ vanishes as well.

\subsubsection*{Step 2: Application of the Hartshorne-Serre correspondence}

All prerequisites satisfied, Fact~\ref{fact:arrondo1} identifies
$$
\Ext¹\bigl( \sI_Y\otimes \sO_X(d), \, \sO_X \bigr) \cong H^0 \bigl( Y,\, \wedge²
\sN_{X/Y}\otimes\sL^\ast|_Y \bigr).
$$
Likewise, the extension class of \eqref{eq:sdf4} is identified with an element
\begin{align*}
  σ_0 ∈ \Ext¹\bigl( \sI_{Y_0} \otimes \sO_{X_0}(d), \, \sO_{X_0} \bigr)
  & \cong H^0 \bigl( Y_0,\, \wedge² \sN_{X_0/Y_0}\otimes \sL^\ast|_{Y_0} \bigr) \\
  & = H^0 \bigl( Y_0,\, (\wedge² \sN_{X/Y}\otimes\sL^\ast)|_{Y_0} \bigr)
\end{align*}
that, by Theorem~\ref{thm:arrondo2}, generates $\wedge² \sN_{X_0/Y_0}\otimes
\sL^\ast|_{Y_0}$.  To conclude, it will therefore suffice to find a section $σ ∈
H^0 \bigl( Y,\, \wedge² \sN_{X/Y}\otimes\sL^\ast|_Y \bigr)$, which generates
$\wedge² \sN_{X/Y}\otimes\sL^\ast|_Y$ and restricts to $σ_0$.  Since $\wedge²
\sN_{X/Y}\otimes\sL^\ast|_Y$ is isomorphic to the trivial sheaf $\sO_Y$, this is
easily possible.  \qed

\subsection{A Bertini-type theorem}
\label{subsec:bertini}

Generalizing the classical Bertini theorem, Kleiman gave conditions guaranteeing
that the zero locus of a sufficiently general section of a vector bundle is
non-singular.  We state a version of Kleiman's Bertini theorem here; our
formulation is quoted from a paper of Hartshorne \cite[Prop.~1.4]{HartshorneSB}.

\begin{fact}[{Bertini-type theorem for sections in vector bundles, \cite[Cor.~3.6]{Kleiman}}]\label{fact:bertini}
  Let $\sE$ by any rank-two vector bundle on $\bP^n$, for $n ≥ 2$.  If $\sE(-1)$
  is generated by global sections, then for all sufficiently general $s ∈ H^0
  \bigl( \bP^n,\, \sE \bigr)$, the associated scheme of zeros is non-singular.
  \qed
\end{fact}

\subsection{Uniqueness of bundle structure}
\label{sec:2B}

The goal of this subsection is to establish Theorem~\ref{thm:uniq}, which shows
that the $\bP¹$-bundle structure on the projectivization of a rank-two bundle on
$\bP²$ is often unique.  In order to state the result, we need to recall the
following definition.  This notion was studied by Strømme, \cite{stroemme}, and
will reappear in later sections.

\begin{defn}[\protect{Type of a bundle on $\bP²$, \cite[Sect.~1.1]{stroemme}}]\label{def:type}
  If $\sE$ is a vector bundle on $\bP²$, set
  $$
  d(\sE) :=
  \begin{cases}
    -1 & \text{if $\sE$ is slope-stable}\\
    \max \bigl\{ d \,|\, H^0 \bigl(\bP²,\, \sE(-d) \bigr) > 0 \bigr\} & \text{otherwise.}
  \end{cases}
  $$
  The number $d(\sE)$ is called the \emph{generic splitting type} of $\sE$, and
  $\sE$ will be said to be ``of \emph{type $d$}''.
\end{defn}

\begin{thm}[Uniqueness of bundle structure]\label{thm:uniq}
  Fix two numbers $c_1 ∈ \{0,-1\}$ and $c_2 ∈ \bZ$ and let $d$ be any number
  such that $d > 3+c_1$.  Let $\sE$ be any rank-two vector bundle on $\bP²$ with
  Chern classes $c_1$ and $c_2$ and type $d$, and let $π : \bP_{\bP²}(\sE) →
  \bP²$ be the obvious bundle map.  Given any other morphism $φ :
  \bP_{\bP²}(\sE) → \bP²$ that has the structure of a Zariski locally trivial
  $\bP¹$-bundle, there exists an automorphism $ψ : \bP² → \bP²$ fitting into a
  commutative diagram of the form:
  $$
  \xymatrix{ %
    \bP_{\bP²}(\sE) \ar@{=}[r] \ar[d]_{φ} & \bP_{\bP²}(\sE) \ar[d]^{π}\\
    \bP² \ar[r]_{ψ} & \bP².
  }
  $$
\end{thm}

\subsection*{Proof of Theorem~\ref*{thm:uniq}}

We prove Theorem~\ref{thm:uniq} in the remaining part of Section~\ref{sec:2B}.
For the reader's convenience, the proof is subdivided into a number of
relatively independent steps.

\begin{figure}
  \centering
  \footnotesize

  \begin{tikzpicture}[scale=0.5]
    \fill[fill=gray!20!white] (0,0) -- (10,5) -- (10,10) -- (5,10);

    \draw (6.8,8.5) node[below]{$\overline{NE_1(X)_{\bR}}$};
    \draw (6.8,7.6) node[below]{cone of effective cycles};
    \draw (6.8,7.0) node[below]{(does not contain a line)};

    \draw [->] (-1,0) -- (11,0) ;
    \draw [->] (0,-1) -- (0,11) ;
    \draw (-2.7,10.9) node[below]{$N_1(X)_{\bR}$};
    \draw (-2.7,10.2) node[below]{numerical classes};
    \draw (-2.7, 9.5) node[below]{of 1-cycles};

    \draw[dashed] (0,0) -- (10,5);
    \fill (2,4) circle (.1) node[left]{$[F_{ψ}]$};
    \draw [->] (2.1,9) node[above]{\quad extremal ray $α_1$} to (3.5,7.5);

    \draw[dashed] (0,0) -- (5,10);
    \fill (5,2.5) circle (.1) node[below]{\quad$[F_{π}]$};
    \draw [->] (9,1.5) node[below]{extremal ray $α_2$} to (7,3.4);
  \end{tikzpicture}

  \bigskip

  {\small The figure illustrates the vector space $N_1(X)_{\bR}$ of numerical
    curve classes that appears in the proof of Theorem~\ref{thm:uniq}.  The
    closed cone $\overline{NE_1(X)_{\bR}}$, which is spanned by effective
    cycles, does not contain a line and therefore has exactly two extremal rays,
    $α_1$ and $α_2$.  Under the assumptions made in the proof, it will turn out
    that these rays are generated by numerical classes of fibers of the bundles
    $π$ and $ψ$, respectively.}

  \caption{Proof of Theorem~\ref{thm:uniq}}
  \label{fig:main}
\end{figure}

\subsubsection*{Step 1.  Setup}

Since $\bP²$ is normal, the claim of Theorem~\ref{thm:uniq} follows from
Zariski's main theorem as soon as we show that any $φ$-fiber $F$ is also a fiber
of $π$.  Since fibers of $π$ are characterized as those curves that intersect
$c_1 \bigl( π^* \sO_{\bP²}(1) \bigr)$ trivially, it suffices to show that the
numerical classes of $π$-fibers and $φ$-fibers agree up to multiplication with a
positive constant.  We argue by contradiction and assume that this is not the
case.  Using standard arguments of minimal model theory, we will see in Step~2
that this assumption implies that $X := \bP_{\bP²}(\sE)$ is Fano, that is, that
the anti-canonical divisor $-K_X$ is ample.  Step~3 then shows that the
numerical assumptions made in Theorem~\ref{thm:uniq} are incompatible with the
Fano property.

\subsubsection*{Step 2.}

The Picard-number of $X$ being two, it follows from the Theorem of the Base of
Néron-Severi, \cite[II Thm.~4.5 and references there]{K96}, that the vector
space of numerical curve classes, $N_1(X)_{\bR}$, is likewise two-dimensional.
Given any ample divisor $D$ on $X$, recall from Kleiman's ampleness criterion,
\cite[IV Thm.~2.19]{K96}, that any numerical class $α$ contained in the closure
of the cone of effective cycles, $\overline{NE_1(X)_{\bR}}$, intersects $D$
positively, $D.α > 0$.  In particular, the cone $\overline{NE_1(X)_{\bR}}$ does
not contain any lines.  As it is convex by definition,
$\overline{NE_1(X)_{\bR}}$ is spanned by two extremal classes, say $α_1$ and
$α_2$.

Intersection with $c_1(π^* \sO_{\bP²}(1))$ defines a non-trivial form on
$N_1(X)_{\bR}$, which is non-negative on $\overline{NE_1(X)_{\bR}}$ and trivial
on the ray $\bR^+ \cdot [F_{π}]$ spanned by the numerical class of any $π$-fiber
$F_{π}$.  It follows that this ray must be one of the two extremal rays of
$\overline{NE_1(X)_{\bR}}$.  The same holds for the numerical class of any
$φ$-fiber $F_{φ}$.  Using the assumption that the numerical classes $[F_{π}]$
and $[F_{φ}]$ are no positive multiples of each other, we have thus identified
$\overline{NE_1(X)_{\bR}}$ as the cone spanned by these two classes,
$$
\overline{NE_1(X)_{\bR}} = \bR^{≥ 0} \cdot [F_{π}] + \bR^{≥ 0} \cdot [F_{φ}].
$$
This observation has further consequences.  Using the $\bP¹$-bundle structure of
$π$ and $φ$, it follows from the adjunction formula that
$$
-K_X.F_{π} = -K_X.F_{φ} = 2.
$$
It follows that $-K_X \cdot C > 0$ for any class $C ∈ \overline{NE_1(X)_{\bR}}
\setminus \{ 0\}$ and thus by Kleiman's ampleness criterion, we conclude that
$-K_X$ is ample.  In other words, $X$ is Fano.

\subsubsection*{Step 3.}

In order to derive a contradiction, we will now construct a curve $C ⊂ X$ which
intersects $-K_X$ negatively.  To this end, we choose a general line $\ell ⊂
\bP²$.  A classical result of Dedekind and Weber \cite{DW82}, often attributed
to Grothendieck, allows us to write $\sE|_{\ell}$ as a sum of line bundles,
$$
\sE|_{\ell} \cong
\begin{cases}
  \sO_{\bP¹}(a) \oplus \sO_{\bP¹}(-a) & \text{if $c_1(\sE) = 0$} \\
  \sO_{\bP¹}(a) \oplus \sO_{\bP¹}(-a-1) & \text{if $c_1(\sE) = -1$,}
\end{cases}
$$
where $a$ is a non-negative integer.  Since $\ell$ is general, it follows
immediately from the definition of generic splitting type that $d ≤ a$.  In
particular, $3 + c_1 < a$.  In either case, $a > 2$.  We obtain that the
preimage of $\ell$ is a Hirzebruch surface of type
$$
π^{-1}(\ell) \cong \bF_b \quad\text{where } b > 4.
$$
Let $C ⊂ \bF_b$, $C \cong \bP¹$ denote the unique section whose
self-intersection equals $-b$.  A two-fold application of the adjunction formula
then shows the following
$$
-K_X.C = \underbrace{c_1( N_{\bF_b/X} ).C}_{= 1} + (-K_{\bF_b}.C) = 1 +
\underbrace{c_1( N_{C/\bF_b} ).C}_{= -b} + \underbrace{\deg T_C}_{=-2} = -b-1
< 0.
$$
This contradicts the result obtained in Step~2 and therefore ends the proof of
Theorem~\ref{thm:uniq}.  \qed

\subsection{Deformation rigidity}
\label{sec:2D}

Assume we are given a proper, surjective morphism of varieties, $X → \bA¹$, and
assume that the fiber $X_0$ over the origin is of the form $X_0 \cong
\bP_{\bP²}(\sE_0)$, for a suitable rank-two vector bundle $\sE_0$ on $\bP²$.
Under favorable conditions, the following Theorem~\ref{thm:rigidity} guarantees
that nearby fibers are also of this form, $X_t \cong \bP_{\bP²}(\sE_t)$, and
that the bundles $\sE_t$ vary smoothly over $\bA¹$.

\begin{thm}[Deformation rigidity of $\bP¹$-bundles over $\bP²$]\label{thm:rigidity}
  Let $f : X → \bA¹$ be a proper, surjective morphism of abstract varieties
  defined over $k$.  Write $X_0 := f^{-1}(0)$ for the scheme theoretic fiber
  over $0$ of $f$.  Assume that there exists a locally free sheaf $\sE_0$ of
  rank two on $\bP²$ and an isomorphism $φ_0 : X_0 → \bP_{\bP²}(\sE_0)$.  Write
  $X_{\reg}$ for the regular locus of $X$ and assume further that the natural
  restriction map $\Pic(X_{\reg}) → \Pic(X_0)$ is surjective.  Then, there
  exists a Zariski-open neighborhood $U = U(0) ⊆ \bA¹$ such that all fibers
  $(X_t)_{t ∈ U}$ are of the form $X_t \cong \bP_{\bP²}(\sE_t)$.

  More precisely, there exist a rank-three, locally free sheaf $\sF_U$ on $U$, a
  rank-two locally free sheaf $\sE_U$ on $Y_U := \bP_U(\sF_U)$ and a commutative
  diagram of the form
  $$
  \xymatrix{ %
    X_0 \ar[dd]_{f|_{X_0}} \ar[r]^(.45){φ_0}_(.45){\cong} & \bP(\sE_0) \ar[rrr]^{\text{closed immersion}} \ar[d]^{\text{$\bP¹$-bundle}} &&& \bP_{Y_U}(\sE_U) \ar[d]_{\text{$\bP¹$-bundle}}^{α_U} & \ar[l]^(.4){\cong}_(.4){φ_U} X_U \ar[rrr]^{\text{open immersion}} \ar[dd]^{f|_{X_U}} &&& X \ar[dd]^f \\
    & \bP² \ar[rrr]_{\text{closed immersion}} \ar[d] &&& Y_U \ar[d]_{\text{$\bP²$-bundle}}^{β_U} \\
    \{ 0 \} \ar@{<->}[r]_{=} & \{ 0 \} \ar[rrr]_{\text{closed immersion}} &&& U \ar@{<->}[r]_{=} & U \ar[rrr]_{\text{open immersion}} &&& \bA¹, }
  $$
  where $X_U := f^{-1}(U)$.
\end{thm}

\begin{rem}[Smoothness of $X$ near $X_0$]\label{rem:2}
  Since $\bA¹$ is one-dimensional and smooth, it follows that the morphism $f$
  of Theorem~\ref{thm:rigidity} is flat, \cite[III Prop.~9.7]{Ha77}.  The
  assumption that $X_0 \cong \bP_{\bP²}(\sE_0)$ therefore implies the existence
  of an open neighborhood $V = V(0) ⊆ \bA¹$ such that $X_V := f^{-1}(V)$ and
  $f|_{X_V}$ are smooth, \cite[III Ex.~10.2]{Ha77}.  The restriction map
  $\Pic(X_{\reg}) → \Pic(X_0)$ used in Theorem~\ref{thm:rigidity} is therefore
  well-defined.
\end{rem}

\subsection*{Proof of Theorem~\ref*{thm:rigidity}}

As before, the proof of Theorem~\ref{thm:rigidity} spans the rest of the present
Section~\ref{sec:2D}.

\subsubsection*{Step 1.  Choices and identifications}

Choose a rank-two locally free sheaf $\sE_0$ on $\bP²$ and one identification
$φ_0 : X_0 → \bP_{\bP²}(\sE_0)$.  With these choices made, consider the natural
projection morphism $η_0 : X_0 → \bP²$ and the invertible sheaves $\sA_0 :=
\sO_{\bP_{\bP²}(\sE_0)}(1)$ and $\sB_0 := η_0^*\bigl( \sO_{\bP²}(1) \bigr)$.
Using the assumption that the natural restriction map $\Pic(X_{\reg}) →
\Pic(X_0)$ is surjective, choose invertible sheaves $\sA$, $\sB$ on $X_{\reg}$
whose restrictions to $X_0$ agree with $\sA_0$ and $\sB_0$, respectively.
Finally, choose an open neighborhood $U = U(0) ⊆ \bA¹$ of the point $0 ∈ \bA¹$
such that $f$ is smooth over $U$.

With the exception of $U$, maintain the choices made in this section throughout
the proof.  For simplicity, we will abuse notation slightly and shrink the
neighborhood $U$ several times in the proof, whenever it becomes clear that
there exists a sub-neighborhood $U' ⊆ U$ where some desirable property holds.

\subsubsection*{Step 2.  Notation}

If $V ⊆ \bA¹$ is any open set, denote the $f$-preimage of $V$ by $X_V :=
f^{-1}(V) ⊆ X$.  If $X_V$ is smooth, denote the restriction of $\sA$ by $\sA_V
:= \sA|_{X_V}$, similarly for $\sB$.  The restriction of $f$ to $V$ is written
as $f_V : X_V → V$.  In a similar vein, if $t ∈ \bA¹$ is any closed point, write
$X_t := f^{-1}(t)$ and $\sA_t := \sA|_{X_t}$, etc.

To avoid awkward notation, write $Y_0 := \bP²$ when thinking of $\bP²$ as the
base of the $\bP¹$-bundle $η_0$.  Fibers of $η_0$ will always be denoted by
$\ell$.

\subsubsection*{Step 3.  Observations}

Semicontinuity of the flat, proper morphism $f$, \cite[Cor.~on
p.~50]{MR2514037}, guarantees that there exists an open neighborhood $V = V(0) ⊆
U$ such that $(f_V)_*(\sO_{X_V}) = \sO_V$.  In particular, fibers of $f_V$ will
be connected.  Shrinking $U$, if necessary, we assume that this holds true on
all of $U$.

\begin{asswlog}\label{ass:conn}
  All fibers of the morphism $f_U : X_U → U$ are connected and
  $(f_U)_*(\sO_{X_U}) = \sO_U$.
\end{asswlog}

\subsubsection*{Step 4.  Construction of $Y_U$}

We will show in this step that the push-forward of the sheaf $\sB_U$ is locally
free.  The space $Y_U$ will be constructed as the projectivization of this
sheaf.

\begin{claim}\label{claim:n1}
  The cohomology groups $H^i \bigl( X_0,\, \sB_0 \bigr)$ vanish, for all $i ∈
  \bN^+$.
\end{claim}
\begin{proof}[Proof of Claim~\ref{claim:n1}]
  Let $\ell ⊂ X_0$ be any fiber of $η_0$.  Then $\ell \cong \bP¹$, the sheaf
  $\sB_0|_{\ell}$ is isomorphic to $\sO_{\bP¹}$, and $h^i\bigl( \ell,\,
  \sB_0|_{\ell}\bigr) = 0$ for all $i ∈ \bN^+$.  In particular, $\bR^i (η_0)_*
  \sB_0 = 0$ for all $i ∈ \bN^+$, \cite[Cor.~2 on p.~50]{MR2514037}.  Given any
  specific number $i ∈ \bN^+$, the cohomology group in question is thus computed
  as follows,
  \begin{align*}
    H^i \bigl( X_0,\, \sB_0 \bigr) & = H^i \bigl( Y_0,\, (η_0)_* \sB_0 \bigr) && \text{Leray spectral sequence, \cite[III Ex.~8.1]{Ha77}}\\
    & = H^i \bigl( \bP²,\, \sO_{\bP²}(1) \bigr) && \text{Definition of $\sB_0$}\\
    & = 0.  && \text{Cohomology of $\bP^n$, \cite[III Thm.~5.1]{Ha77}.}
  \end{align*}
  This finishes the proof of Claim~\ref{claim:n1}.
\end{proof}

\begin{claim}\label{claim:n2}
  There exists an open, affine neighborhood $V = V(0) ⊆ U$ with the following
  properties:
  \begin{enumerate}
  \item\label{il:1} The sheaf $(f_V)_* \sB_V$ is locally free of rank three.
  \item\label{il:2} Given any closed point $t ∈ V$, let $k(t)$ denote the
    associated residue field.  With this notation, the natural maps $(f_* \sB)
    \otimes_{\sO_V} k(t) → H^0 \bigl( X_t,\, \sB_t \bigr)$ are isomorphisms, for
    all closed points $t ∈ V$.
  \item\label{il:3} The natural restriction map $r_t : H^0 \bigl( X_V,\, \sB_V
    \bigr) → H^0 \bigl( X_t,\, \sB_t \bigr)$ is surjective, for all closed
    points $t ∈ V$.
  \end{enumerate}
\end{claim}
\begin{proof}[Proof of Claim~\ref{claim:n2}]
  Recall the following standard continuity and semicontinuity properties of the
  flat, proper morphism $f$, \cite[Cor.~on p.~50]{MR2514037}.
  \begin{enumerate}
  \item\label{il:a1} The functions $φ_i : U → \bN$, $t \mapsto h^i\bigl( X_t,\,
    \sB_t \bigr)$ are upper semicontinuous for all $i ∈ \bN$.
  \item\label{il:a2} The function $χ : U → \bN$, $t \mapsto \sum_{i ∈ \bN}
    (-1)^i φ_i(t)$ is constant.
  \end{enumerate}
  Claim~\ref{claim:n1} and Item~\ref{il:a1} imply the existence of an open,
  affine neighborhood $V = V(0) ⊆ U$ such that $φ_i(t) = 0$ for all closed
  points $t ∈ V$ and all indices $i ∈ \bN^+$.  Together with Item~\ref{il:a2},
  we see that $χ = φ_0$ is constant on $V$.  By \cite[Cor.~2 on
  p.~50]{MR2514037}, this already implies that $f_* \sB|_V$ is locally free and
  that \ref{il:2} holds.  As for \ref{il:1}, the rank of $(f_V)_* \sB_V$ is
  computed as follows,
  \begin{align*}
    \rank \bigl( (f_V)_* \sB_V \bigr) & = h^0 \bigl( X_0,\, \sB_0 \bigr) && \text{Isomorphism~\ref{il:2} in case $t=0$} \\
    & = h^0 \bigl( X_0,\, (η_0)^* \sO_{\bP²}(1) \bigr) = 3.  && \text{Definition of $\sB_0$.}
  \end{align*}
  Surjectivity of $r_t$, as asserted in \ref{il:3}, follows because $V$ was
  taken to be affine.  This finishes the proof of Claim~\ref{claim:n2}.
\end{proof}

To simplify notation, we shrink $U$ if necessary, and assume the following.

\begin{asswlog}\label{awlog:1}
  Items~\ref{il:1}--\ref{il:3} of Claim~\ref{claim:n2} hold on $U$.
\end{asswlog}

Construct $Y_U$ as a $\bP²$-bundle over $U$ by setting $\sF_U := (f_U)_* \sB_U$
and $Y_U := \bP_U(\sF_U)$.  Maintain these choices for the remainder of the
proof.

\subsubsection*{Step 5.  Factorization of $f$}

In this step, it will be shown that the morphism $f_U$ factorizes via $Y_U$.
The following claim will be important.

\begin{claim}\label{claim:n3}
  There exists an open neighborhood $V = V(0) ⊆ U$ such that $f$ is smooth over
  $V$ and such that the natural evaluation morphism,
  $$
  e : (f_U)^* (f_U)_* \sB_U → \sB_U,
  $$
  is surjective on $X_V$.
\end{claim}
\begin{proof}[Proof of Claim~\ref{claim:n3}]
  Let $\Bs(\sB_U) ⊂ X_U$ be the base point locus of the sheaf $\sB$ on $X_U$.
  More precisely, let $\Bs(\sB_U)$ be the support of $\coker(e)$, with its
  natural structure as a proper closed, reduced subscheme of $X_U$.  We claim
  that $\Bs(\sB_U)$ does not intersect the fiber $X_0$, that is, $\Bs(\sB_U) ∩
  X_0 = \emptyset$.  Once this is shown, set
  $$
  V := U \setminus f_U \bigl( \Bs(\sB_U) \bigr).
  $$
  Since $f_U$ is proper, this will be an open neighborhood of $0 ∈ U$ with all
  desired properties.

  In order to prove the claim, it suffices to show that the natural restriction
  $$
  r_x : H^0 \bigl( X_U,\, (f_U)^* (f_U)_* \sB_U \bigr) → H^0 \bigl( \{x\}, \sB_U|_{\{x\}} \bigr)
  $$
  is surjective, for any closed point $x ∈ X_0$.  However, given any such $x$,
  observe that the morphism $r_x$ factors as follows,
  $$
  \xymatrix{ %
    H^0 \bigl( X_U,\, (f_U)^* (f_U)_* \sB_U \bigr) \ar@/^0.2cm/[drrr]^{r_x} \ar[d]_{φ}^{\text{isomorphism}}\\
    H^0 \bigl( X_U,\, \sB_U \bigr) \ar@{->>}[r]^(.45){r_1}_(.45){\text{restr.~to $X_0$}} & H^0 \bigl( X_0, \sB_U|_{X_0} \bigr) \ar@{->>}[rr]^(.45){r_2}_(.45){\text{restr.~to $\{x\}$}} && H^0 \bigl( \{x\}, \sB_U|_{\{x\}} \bigr).
  }
  $$
  In the diagram above, the morphism $φ$ is the inverse of the natural map $H^0
  \bigl( X_U,\, \sB_U \bigr) → H^0 \bigl( X_U,\, (f_U)^* (f_U)_* \sB_U \bigr)$,
  which is isomorphic because the fibers of $f_U$ are connected by
  Assumption~\ref{ass:conn}.  Surjectivity of $r_1$ holds by
  Assumption~\ref{awlog:1}.  Surjectivity of $r_2$ holds by choice of
  $\sB_U|_{X_0} = \sB_0$.  It follows that $r_x$ is surjective.  This finishes
  the proof of Claim~\ref{claim:n3}.
\end{proof}

As before, we shrink $U$ if necessary, and assume the following.

\begin{asswlog}\label{awlog:3}
  The evaluation morphism $e$ is surjective on $X_U$.
\end{asswlog}

Recall from \cite[II Prop.~7.12]{Ha77} that to give a morphism $X_U → Y_U =
\bP_U( \sF_U )$ over $U$, it is equivalent to give an invertible sheaf $\sL$ on
$X_U$ and a surjective map of sheaves $(f_U)^*(\sF_U) = (f_U)^* (f_U)_* \sB_U →
\sL$.  Setting $\sL := \sB_U$, the evaluation map $e$ considered above therefore
gives rise to a factorization of $f_U$,
$$
\xymatrix{
  X_U \ar[rr]_{η_U} \ar@/^4mm/[rrrr]^{f_U} && Y_U \ar[rr]_{β_U\text{, $\bP²$-bundle}} && U.
}
$$

\subsubsection*{Step 6.  The central fiber of $β_U$}

We claim that the fiber $F := β_U^{-1}(0)$ is canonically isomorphic to $Y_0
\cong \bP²$, and that this isomorphism identifies the restricted map $η_U|_{X_0}
: X_0 → F$ with the projection map $η_0 : X_0 → Y_0$.  With these
identifications, our choice of notation is consistent: $η_0 = η_U|_{X_0}$ and
$Y_0 = β_U^{-1}(0)$.

Both claims follow from compatibility of $\sProj$ and base change,
\cite[Prop.~3.5.3]{EGA2}.  More precisely,
\begin{align*}
  F = β_U^{-1}(0) & = \sProj \Sym \bigl(\sF_U \otimes_{\sO_U} k(0) \bigr) && \text{Base change} \\
  & = \Proj \Sym H^0 \bigl( X_0,\, \sB_0 \bigr) && \text{Claim~\ref{claim:n2}, Item~\ref{il:2}} \\
  & = \Proj \Sym H^0 \bigl( X_0,\, (η_0)^*\sO_{Y_0}(1) \bigr) && \text{Definition of $\sB_0$.}
\end{align*}

\subsubsection*{Step 7.  Fibers of the morphism $η_U$}

\begin{claim}[Smoothness of $η$]\label{claim:n4}
  There exists an open neighborhood $V = V(0) ⊆ U$ such that $η_V$ is smooth
  over $Y_V$.
\end{claim}
\begin{proof}[Proof of Claim~\ref{claim:n4}]
  Let $B ⊂ X_U$ be the closed set where the morphism $η_U$ is \emph{not} smooth.
  We claim that $B$ does not intersect the fiber $X_0$, that is, $B ∩ X_0 =
  \emptyset$.  Once this is shown, set
  $$
  V := U \setminus f_U( B ).
  $$
  Since $f_U$ is proper, this will be an open neighborhood of $0 ∈ U$ with all
  desired properties.

  In order to establish the claim, let $x ∈ X_0$ be any closed point.  We will
  show that $η_U$ is smooth at $x$ by using the criterion \cite[II
  Cor.~2.2]{SGA1}: the morphism $η_U$ is smooth at $x$ if $f_U = β_U ◦ η_U$ is
  smooth at $x$, and if the restriction of $η_U$ to the fibers, $η_U|_{X_0} :
  X_0 → η_U^{-1}(0)$ is smooth.  Smoothness of $f_U$ at $x$ holds by
  Remark~\ref{rem:2}.  Smoothness of $η_U|_{X_0}$ has been established in Step 6
  above.  This finishes the proof of Claim~\ref{claim:n4}.
\end{proof}

Claim~\ref{claim:n4} and the same reasoning as in Step~3 allow to make the
following additional assumptions.

\begin{asswlog}\label{ass:c2}
  The morphism $η_U$ is smooth.  Its fibers are connected.
\end{asswlog}

\begin{claim}\label{claim:n5}
  If $y ∈ Y_U$ is any closed point with associated fiber $X_y := η_U^{-1}(y)$,
  then $X_y \cong \bP¹$.
\end{claim}
\begin{proof}[Proof of Claim~\ref{claim:n5}]
  Assumption~\ref{ass:c2} implies that the fibers of $η_U$ are complete, smooth,
  connected curves.  As before, \cite[Cor.~on p.~50]{MR2514037}, guarantees
  that the function
  $$
  χ : Y_U → \bN, \quad y \mapsto \sum_{i ∈ \bN} (-1)^i h^i\bigl( X_y, \,
  \sO_{X_y} \bigr)
  $$
  is constant on $Y_U$.  Since $χ(y) = 1-g(X_y)$ for any closed point $y ∈ Y_U$
  and since $X_y \cong \bP¹$ if $y ∈ X_0$, it follows that all fibers of $η_U$
  are isomorphic to $\bP¹$.  This finishes the proof of Claim~\ref{claim:n5}.
\end{proof}

\subsubsection*{Step 8.  End of proof}

To end the proof, we need to show that the smooth morphism $η_U$ has the
structure of a $\bP¹$-bundle.  Since all its fibers are isomorphic to $\bP¹$ and
since the invertible sheaf $\sA$ has degree one on each fiber, this follows
quickly from arguments that are quite similar to those used in Steps~4 and 5.
For projective morphisms between complex varieties, everything has been shown in
\cite[Lem.~2.12]{MR946238}.

We aim to construct an explicit $\bP¹$-bundle which will then turn out to be
isomorphic to $X_U$.  To this end, set $\sE_U := (η_U)_*(\sA)$ and observe that
$\sE_U|_{Y_0} \cong \sE_0$ by choice of $\sA$.  Since all fibers $X_y$ are
isomorphic to $\bP¹$ and since the invertible sheaf $\sA$ has degree one on
these fibers, it follows that the function
$$
φ : Y_U → \bN, \quad y \mapsto h^0\bigl( X_y, \, \sA_U|_{X_y} \bigr)
$$
is constant of value two.  As before, invoke \cite[Cor.~2 on p.~50]{MR2514037}
to conclude that $\sE_U$ is locally free of rank two.  Using that $\sA_U|_{X_y}$
is identified with $\sO_{\bP¹}(1)$ and is hence basepoint-free for any closed
point $y ∈ Y_U$, a minor variant of the argumentation used in the proof of
Claim~\ref{claim:n3} reveals that the evaluation map
$$
(η_U)^* (η_U)_* \sA_U → \sA_U
$$
is surjective.  As before, we have thus constructed a refined factorization of
$f_U$,
$$
\xymatrix{
  X_U \ar[rr]_(.45){\phantom{\bP²}φ_U\phantom{\bP²}} \ar@/^4mm/[rrrrrr]^{f_U} && \bP_{Y_U}(\sE_U) \ar[rr]_(.55){α_U\text{, $\bP¹$-bundle}} && Y_U \ar[rr]_{β_U\text{, $\bP²$-bundle}} && U.
}
$$
By construction, the restriction of the $φ_U$ to any fiber $X_y$ is identified
with the morphism induced by the very ample invertible sheaf $\sA_U|_{X_y} \cong
\sO_{\bP¹}(1)$, that is,
$$
\bP¹ → \bP \Bigl( H^0\bigl( \bP¹,\, \sO_{\bP¹}(1) \bigr) \Bigr).
$$
This has two consequences.  First, the smoothness criterion \cite[II
Cor.~2.2]{SGA1} applies to show that $φ_U$ is smooth.  In particular, $φ_U$ is
separable, \cite[Chapt.~AG, Thm.~17.3]{B91}.  Second, it follows that the
morphism $φ_U$ is bijective.  By \cite[Sect.~2]{Gro5658} or \cite[Thm.~on
p.~43]{B91}, the induced morphism between functions fields has separable degree
equal to one.  It follows that $φ_U$ is birational.  Since all spaces in
question are smooth, hence normal, Zariski's Main Theorem,
\cite[Lem.~8.12.10.1]{EGA4-3}, therefore guarantees that $φ_U$ is isomorphic.
This finishes the proof of Theorem~\ref{thm:rigidity}.  \qed

\subsection{Deformations and moduli}
\label{ssec:deformation}

We recall Strømme's results on moduli of vector bundles and draw first
conclusions concerning deformability and non-deformability of vector bundles.

\subsubsection{Notation and known facts}
\label{ssec:5B1}

Projectivizations of rank-two vector bundles are the main objects of interest in
this paper.  In the discussion, we will often be free to twist any given vector
bundle with a suitable line bundle, allowing to assume that the bundle's first
Chern class is either zero or minus one.

\begin{setting}[Choice of Chern classes]\label{setting:stroemme}
  Fix two numbers $c_1 ∈ \{0,-1\}$ and $c_2 ∈ \bZ$.
\end{setting}

\begin{defn}[Families of bundles]\label{def:families}
  Let $T$ be a $k$-scheme.  Given numbers $c_1$ and $c_2$, a \emph{family}
  $\mathfrak{E}/T$ of rank-two vector bundles on $\bP²$ with Chern classes $c_1$
  and $c_2$ is a rank-two bundle $\sE$ on $T ⨯ \bP²$ such that for any
  $k$-valued point $t ∈ T$, the fiber $\sE_t$ is a rank-two bundle on $\bP²$,
  with Chern classes $c_1(\sE_t) = c_1$ and $c_2(\sE_t) = c_2$.
\end{defn}

\begin{defn}[\protect{Pure type, \cite[Sect.~2.3]{stroemme}}]\label{def:ptfamilies}
  In the setting of Definition~\ref{def:families}, given any integer $d ≥ 0$,
  the family $\mathfrak{E}/T$ is said to be of \emph{pure type $d$}, if $\bR² π_*
  \bigl( \sE^*(d-3)\bigr)$ is invertible, where $π : T ⨯ \bP² → T$ is the
  natural projection map.
\end{defn}

\begin{fact}[\protect{Type and pure type, \cite[Rem.~2.4]{stroemme}}]
  In the setting of Definition~\ref{def:ptfamilies}, if $\mathfrak{E}/T$ is of
  pure type $d$, then all bundles $\sE_t$ have generic splitting type $d$.
  \qed
\end{fact}

\begin{fact}[\protect{Semicontinuity, \cite[Sect.~2.2]{stroemme}}]\label{fact:semincontST}
  In the setting of Definition~\ref{def:families}, the generic splitting type is
  upper semicontinuous as a function on the closed points of $T$.  Given any $d
  ≥ 0$, there exists a maximal, locally closed subscheme $T(d) ⊆ T$
  over which the bundle is of pure type $d$.  \qed
\end{fact}

\begin{fact}[\protect{Existence of moduli spaces, \cite[Props.~1.2 and 2.7]{stroemme}}]\label{fact:exMd}
  Given numbers $c_1$ and $c_2$ as in Setting~\ref{setting:stroemme}, let $d ≥
  0$ be any number.  Then, there exists a coarse moduli scheme $M(d)$ for
  families of rank-two vector bundles on $\bP²$ of pure type $d$, modulo
  isomorphism and twists by line bundles coming from the base.  The dimension of
  $M(d)$ is computed as follows.
  \begin{enumerate}
  \item If $d²-dc_1+c_2 < 0$, then $M(d)$ is empty.
  \item If $d²-dc_1+c_2 = 0$, then $M(d)$ is a point.
  \item If $d²-dc_1+c_2 > 0$, then $\dim M(d) = 3(d²-dc_1+c_2)-1$.
  \end{enumerate}
  The scheme $M(d)$ is either empty, or irreducible, nonsingular,
  quasiprojective and rational.  \qed
\end{fact}

\begin{fact}[\protect{Existence of maximal families, \cite[Sect.~3.1--3.6 and Thm.~3.9]{stroemme}}]\label{fact:exFam}
  Given numbers $c_1$ and $c_2$ as in Setting~\ref{setting:stroemme} and $d ≥
  -1$, then there exists a smooth, irreducible scheme $Q(d)$ and a family
  $\mathfrak{E}/Q(d)$ of rank-two vector bundles on $\bP²$ with Chern classes
  $c_1$ and $c_2$, such that $\mathfrak{E}/Q(d)$ is pure type $d$ and such that
  the induced moduli map $Q(d) → M(d)$ is surjective.  \qed
\end{fact}

\begin{defn}[\protect{Deformability to given type over irreducible base, \cite[Sect.~2.12 and Thm.~3.13]{stroemme}}]\label{def:X}
  Given numbers $c_1$ and $c_2$ as in Setting~\ref{setting:stroemme}, numbers $d
  > e ≥ 0$, and a rank-two vector bundle $\sE$ on $\bP²$ of type $d$.  We say
  that $\sE$ is \emph{deformable to type $e$ over an irreducible base} if there
  exists an irreducible $k$-scheme $T$ and a family of bundles $\mathfrak{E}/T$
  with Chern classes $c_1$ and $c_2$ that is generically of type $e$ and
  contains $\sE$ as a fiber.
\end{defn}

\begin{fact}[\protect{Locus of deformable bundles, \cite[Sect.~3.12 and Thms.~3.13, 4.7]{stroemme}}]\label{fact:locDef}
  Given numbers $c_1$ and $c_2$ as in Setting~\ref{setting:stroemme}, and $d > e
  ≥ -1$.  Then, there exists a closed subset $M(d;e) ⊆ M(d)$ whose $k$-rational
  points are exactly those isomorphism classes of bundles that are deformable to
  type $e$ over an irreducible base.  If $M ⊆ M(d;e)$ is any irreducible
  component, then $\codim_{M(d)} M ≥ γ(d;e)$, where
  \begin{align*}
    γ(d;e) & :=
    \begin{cases}
      P(d) & \text{if } e=-1 \text{ or } e = c_1 = c_2 = 0 \\
      P(d)-P(e)+1 & \text{otherwise}.
    \end{cases} && \text{and}\\
    P(x) & := (x-1)(x-2-c_1)-c_2.
  \end{align*}
  If $\binom{d-e-1}{2} ≥ e²- e \cdot c_1 + c_2$, then $M(d;e)$ contains an
  irreducible component for which equality holds.  \qed
\end{fact}

\begin{obs}[Numerology]\label{obs:num}
  In the setting of Fact~\ref{fact:locDef}, elementary computations show that if
  $d \gg 0$ is sufficiently large, then $γ(d;e) > 0$ for all numbers $e$
  satisfying $d > e ≥ -1$.  In particular, for any such $e$, the locus $M(d;e)$
  of bundles that are deformable to type $e$ over an irreducible base is either
  empty, or a proper closed subset, $M(d;e) \subsetneq M(d)$.
\end{obs}

\begin{cor}[Non-emptyness of $M(d;e)$]\label{cor:myDRBI}
  Given numbers $c_1$ and $c_2$ as in Setting~\ref{setting:stroemme} and $e ≥
  -1$ such that $M(e)$ is not empty.  If $d \gg 0$ is any sufficiently large
  number, then $M(d;e) \subsetneq M(d)$ is a proper, non-empty subvariety.
\end{cor}
\begin{proof}
  Given $c_1$, $c_2$ and $e$, consider the polynomials $P(\cdot)$ and
  $γ(\cdot;e)$ as in Fact~\ref{fact:locDef}.  If $d \gg 0$ is sufficiently
  large, then any of the following polynomials in $d$, which all have positive
  leading coefficients, takes strictly positive values.
  \begin{align*}
    Q_1(d) & := d² - d\cdot c_1 + c_2 &
    Q_2(d) & := 3(d²- d\cdot c_1 + c_2) - 1 \\
    Q_3(d) & := \textstyle \binom{d-e-1}{2} - e²- e \cdot c_1 + c_2 &
    Q_4(d) & := Q_2(d)-γ(d;e) \\
    Q_5(d) & := γ(d;e)
  \end{align*}
  Fact~\ref{fact:exMd} asserts that $M(d)$ is non-empty as soon as $Q_1(d) ≥ 0$.
  The dimension of $M(d)$ is then given as $Q_2(d)$.  Fact~\ref{fact:locDef}
  claims that once $Q_3(d)$ is positive, the space $M(d;e) ⊆ M(d)$ contains a
  component $M$ whose dimension $\dim M$ is equal to $Q_4(d)$, and therefore
  again positive.  The minimal codimension in $M(d)$ of components of
  $M(d;e)$ is given by $Q_5(d)$, showing that $M(d;e) \ne M(d)$.
\end{proof}

\subsubsection{Deformability and non-deformability}

As a consequence of Observation~\ref{obs:num} we will see in
Proposition~\ref{prop:stroemmeX} that most vector bundles cannot be deformed
over an irreducible base to bundles of smaller type.  In striking contrast, we
will see in Proposition~\ref{prop:explicit1-x} that any two vector bundles whose
Chern classes are equal are deformable into each other, over a base that is not
necessarily irreducible.

\begin{prop}[Non-deformability over irreducible base]\label{prop:stroemmeX}
  Given numbers $c_1$ and $c_2$ as in Setting~\ref{setting:stroemme}.  If $d \gg
  0$ is sufficiently large, then there exists a rank-two vector bundle
  $\sE$ of type $d$ on $\bP²$ with Chern classes $c_1$ and $c_2$ that is not
  deformable to type $e$ over an irreducible base, for any $d > e ≥ -1$, in the
  sense of Definition~\ref{def:X}.
\end{prop}
\begin{proof}
  Recall from Observation~\ref{obs:num} that the open complement $M(d) \setminus
  \bigcup_{d > e ≥ -1} M(d;e)$ is not empty.  Choose a $k$-rational point in
  there and let $\sE$ be the corresponding bundle.
\end{proof}

\begin{prop}[Deformability over reducible base]\label{prop:explicit1-x}
  Given numbers $c_1$, $c_2$ as in Setting~\ref{setting:stroemme} and vector
  bundles $\sA$ and $\sB$ with Chern classes $c_1$, $c_2$.  Then $\sA$ and $\sB$
  are deformable into each other, over a base scheme that need not necessarily
  be irreducible.
\end{prop}
\begin{proof}
  Denote the splitting types of $\sA$ and $\sB$ by $e_{\sA}$ and $e_{\sB}$,
  respectively.  Corollary~\ref{cor:myDRBI} then gives a number $d \gg 0$ such
  that $M(d;e_{\sA})$ and $M(d;e_{\sB})$ are both non-empty.  Choose a bundle
  $\sC ∈ M(d)$.  By Fact~\ref{fact:exFam}, there exists a deformation family
  that connects the bundle $\sC$ to one in $M(d;e_{\sA})$.  By definition of
  $M(d;e_{\sA})$, this bundle can be deformed into one in $M(e_{\sA})$, which,
  by Fact~\ref{fact:exFam} again, can be deformed into $\sA$.  We have thus
  found a deformation over a reducible base that has $\sC$ and $\sA$ as fibers.
  In a similar manner, find a deformation that connects $\sC$ and $\sB$.
  Connect these deformation families to conclude.
\end{proof}

%
%

\section{Homotopy classification of \texorpdfstring{$PGL_2$}{PGL2}-torsors over \texorpdfstring{$\bP²$}{P\texttwosuperior}}
\label{sec:03}

In this section, we discuss the $\bA¹$-homotopy classification of (Nisnevich
locally trivial) $PGL_2$-torsors on $\bP²$.  In other words, we describe the
pointed set $[\bP²,BPGL_2]_{\bA¹}$.  To formulate a useful description of this
set, we observe that Nisnevich locally trivial $PGL_2$-torsors are always
obtained from $GL_2$-torsors by change of structure group.  We then investigate
the induced map
$$
[\bP²,BGL_2]_{\bA¹} \longrightarrow [\bP²,BPGL_2]_{\bA¹},
$$
show that this map is surjective, and describe the right hand side as a quotient
of the left hand side by the natural action of $\Pic(\bP²)$ coming from
``tensoring by line bundles''.  Using an explicit description of
$[\bP²,BGL_2]_{\bA¹}$ that stems from techniques of obstruction theory, we then
obtain a description of $[\bP²,BPGL_2]_{\bA¹}$.  The main results of this
section are Theorem~\ref{thm:pgl2torsors} and
Corollary~\ref{cor:pgl2torsorsonp2}.

\begin{conventions}
  In Sections~\ref{sec:03} and \ref{sec:04}, we deviate from our global
  conventions.  Fix an algebraically closed field $k$.  Contrary to our global
  assumptions fixed in Section~\ref{sec:conventions}, these sections use
  different assumptions on the characteristic of $k$; we will always be explicit
  about the primes we want to exclude.

  We write $\Sm_k$ for the category of schemes that are separated, finite type
  and smooth over $\Spec k$.  We write $\Spc_k$ for the category of simplicial
  Nisnevich sheaves of sets on $\Sm_k$, equipped with the $\bA¹$-local model structure
  of \cite{MV}.  In the rest of this section, the word ``sheaf" will by synonymous with ``Nisnevich sheaf of groups on $\Sm_k$."

  A presheaf $\mathbf{F}$ on $\Sm_k$ is called $\bA¹$-invariant if
  $\mathbf{F}(U) → \mathbf{F}(U ⨯ \bA¹)$ is a bijection for any $U ∈ \Sm_k$.  A
  sheaf of groups $\mathbf{G}$ is strongly $\bA¹$-invariant if its cohomology
  presheaves $H^i \bigl( \cdot,\, \mathbf{G} \bigr)$ are $\bA¹$-invariant, for
  $i ∈ \{ 0,1 \}$.  A sheaf of abelian groups $\mathbf{A}$ is called strictly
  $\bA¹$-invariant if all its cohomology presheaves are $\bA¹$-invariant.  By
  \cite[Thm.~1.9]{MField}, if $(\mathscr{X},x)$ is a pointed space, then
  $\bpi_i^{\bA¹}({\mathscr X},x)$ is strongly $\bA¹$-invariant for $i = 1$, and
  strictly $\bA¹$-invariant for $i > 1$.
\end{conventions}

\begin{rem}
  The classification results of Sections~\ref{sec:03} and \ref{sec:04} hold in
  greater generality: statements and proofs apply verbatim to the case where the
  base field $k$ is quadratically closed.
\end{rem}

\subsection{Torsors and classifying spaces in \texorpdfstring{$\bA¹$}{A\textonesuperior}-homotopy theory}
\label{ss:torsors}

If $\mathscr{G}$ is a sheaf of groups, we write $B\mathscr{G}$ for the
simplicial bar construction of the sheaf of groups $\mathscr{G}$,
\cite[§~4.1]{MV}.  By \cite[Prop.~4.1.15]{MV}, we know that (free) simplicial
homotopy classes of maps from a smooth scheme $X$ to $B\mathscr{G}$ are in
bijection with Nisnevich locally trivial $\mathscr{G}$-torsors on $X$.  Thus, if
$B\mathscr{G}^f$ is a simplicially fibrant model of $B\mathscr{G}$, then, given
a Nisnevich locally trivial $\mathscr{G}$-torsor $π: P → X$, we can pick a
morphism $f_{π}: X → B\mathscr{G}^f$ such that $π$ is the pullback of the
universal $\mathscr{G}$-torsor along $f_{π}$.

The space $B\mathscr{G}$ is a reduced simplicial sheaf (i.e., the sheaf of
$0$-simplices is reduced to a point) and is therefore simplicially
$0$-connected.  It follows from \cite[Cor.~2.3.22]{MV} that $B\mathscr{G}$ is
$\bA¹$-connected.  We write $\ast$ for the canonical base-point of
$B\mathscr{G}$.  If we write $X_+$ for $X$ with a disjoint base-point attached,
then ``forgetting the base-point" induces a bijection between the set of pointed
morphisms from $X_+$ to $B\mathscr{G}$ and the set of morphisms from $X$ to
$B\mathscr{G}$.  In particular, we can always assume that $f_{π}$ is represented
by a pointed morphism from $X_+$.

If $G$ is a linear algebraic group, then $G$ can be viewed as an étale sheaf of
groups, and we can consider the étale classifying space $B_{\et}G$; see
\cite[§~4.2]{MV} for the construction.  There is a canonical adjunction morphism
$BG → B_{\et}G$ that is a simplicial weak equivalence if and only if étale
locally trivial $G$-torsors are Nisnevich locally trivial.

If $G$ is a finite étale group scheme of order coprime to the characteristic of
$k$, then $B_{\et}G$ is $\bA¹$-local by \cite[Prop.~4.3.5]{MV}.  As a
consequence, if $X$ is a smooth scheme, then $[X,B_{\et}G]_{\bA¹}$ is in natural
bijection with the set of étale locally trivial $G$-torsors on $X$.  We define
$\mathcal{H}¹_{\et}(G)$ to be the Nisnevich sheafification of the presheaf $U
\mapsto [U,B_{\et}G]_{\bA¹}$.  If $G$ is abelian, the sheaf
$\mathcal{H}¹_{\et}(G)$ is a sheaf of abelian groups, and under the hypothesis
on $k$, is also strictly $\bA¹$-invariant.  The important fact, used below
without explicit reference, is that morphisms of strictly $\bA¹$-invariant
sheaves are determined by their sections over extensions of the base field.
This follows because such sheaves are unramified in the sense of
\cite[Def.~2.1]{MField}, cf.~\cite[Cor.~6.9 and Rem.~6.10]{MField}.

\subsection{Some \texorpdfstring{$\bA¹$}{A\textonesuperior}-homotopy theory of \texorpdfstring{$PGL_2$}{PGL2}}

In this section, we produce some $\bA¹$-fiber sequences related to $PGL_n$ and
$BPGL_n$.  We refer to \cite{WendtTorsor} for discussion of the general theory
of $\bA¹$-fiber sequences.

\begin{lem}\label{lem:pglnfibseq1}
  There is an $\bA¹$-fiber sequence of the form $PGL_n → B\gm → BGL_n$.
\end{lem}
\begin{proof}
  Write $EGL_n$ for the {\v C}ech simplicial object associated with the
  structure map $GL_n → \Spec k$.  The inclusion of the center $\gm
  \hookrightarrow GL_n$ yields an isomorphism $GL_n/\gm \isomt PGL_n$ and there
  is a natural left translation action of $GL_n$ on $GL_n/\gm$.  Consider the
  associated fiber space $EGL_n ⨯^{GL_n} GL_n/\gm$.  Projection onto the first
  factor gives a morphism $EGL_n ⨯^{GL_n} GL_n/\gm → BGL_n$ that, as the
  associated fiber space of a $GL_n$-torsor, is automatically an $\bA¹$-fiber
  sequence by \cite[Prop.~5.1 and Thm.~5.3]{WendtTorsor}.  On the other hand it
  is straightforward to show that $EGL_n ⨯^{GL_n} GL_n/\gm$ is simplicially
  weakly equivalent to $B\gm$.  This is established in exactly the same fashion
  as the proof of \cite[Lem.~3.8]{AsokPi1}.
\end{proof}

The map $EGL_n ⨯^{GL_n} GL_n/\gm → BGL_n$ in the proof of
Lemma~\ref{lem:pglnfibseq1} is, as the associated fiber space of a
$GL_n$-torsor, Nisnevich locally trivial; under the identification $GL_n/\gm
\cong PGL_n$ this map is furthermore a $PGL_n$-torsor.  As a consequence, there
exists a classifying morphism $BGL_n → BPGL_n$ for this map.  The next result
then follows from \cite[Prop.~5.1 and Thm.~5.3]{WendtTorsor}.

\begin{lem}\label{lem:gmglnpglnfibersequence}
  There is an $\bA¹$-fiber sequence of the form $B\gm → BGL_n → BPGL_n$.  \qed
\end{lem}

The following result is essentially contained in \cite[Cor.~3.17]{AsokPi1} and
\cite[Props.~5.11, 5.12]{chev-rep}, though the formulation and proof below are
somewhat different.

\begin{prop}\label{prop:homotopysheavesofpgl2}
  Let $n≥ 2$ be a natural number, and assume that the base field $k$ has
  characteristic that does not divide $n$.  Then there is a canonical
  isomorphism
  $$
  \bpi_1^{\bA¹}(BPGL_n,\ast) \isomto \mathcal{H}¹_{\et}(μ_n)
  $$
  and a short exact sequence of strictly $\bA¹$-invariant sheaves of the form
  $$
  0 \longrightarrow \bpi_2^{\bA¹}(BGL_n) \longrightarrow \bpi_2^{\bA¹}(BPGL_n)
  \longrightarrow μ_n \longrightarrow 0
  $$
  with $\bpi_2^{\bA¹}(BGL_2)\cong \K^{MW}_2$ and $\bpi_2^{\bA¹}(BGL_n)\cong
  \K^M_2$ for $n≥ 3$.
\end{prop}
\begin{proof}
  The $\bA¹$-fiber sequence
  \begin{equation}\label{eq:ght5}
    B\gm → BGL_n → BPGL_n
  \end{equation}
  induces a long exact sequence in $\bA¹$-homotopy sheaves
  \cite[Lem.~2.10]{AsokPi1}.  There is a canonical isomorphism
  $\bpi_1^{\bA¹}(BGL_n) \isomto \gm$ induced by the determinant homomorphism.
  As described in the proof of Lemma~\ref{lem:pglnfibseq1}, the map $B\gm →
  BGL_n$ in the above $\bA¹$-fiber sequence is induced by the inclusion of the
  center $\gm → GL_n$.  If $t$ is a coordinate on $\gm$, then the composite map
  $\gm → GL_n → \gm$, where the second homomorphism is induced by the
  determinant, is given by $t \mapsto t^n$.  In particular, the map $\gm \cong
  \bpi_1^{\bA¹}(B\gm) → \bpi_1^{\bA¹}(BGL_n) \cong \gm$ in the long exact
  sequence is precisely the map $t \mapsto t^n$.  It follows that
  $\bpi_1^{\bA¹}(BPGL_n)$ is isomorphic to the Nisnevich sheaf quotient
  $\factor{\gm}{\mathbf{G}_m^n}$.

  The Kummer sequence of étale sheaves $μ_n → \gm → \gm$ yields an exact
  sequence of cohomology presheaves
  $$
  \gm(\cdot) \stackrel{⨯ n}{\longrightarrow} \gm(\cdot) \longrightarrow
  H¹_{\et}(\cdot,μ_n) \longrightarrow H¹_{\et}(\cdot,\gm).
  $$
  Sheafifying this sequence of presheaves for the Nisnevich topology on $\Sm_k$,
  and observing that the Nisnevich sheafification of $H¹_{\et}(\cdot,\gm)$ is
  trivial, yields a canonical isomorphism of Nisnevich sheaves
  $\factor{\gm}{\mathbf{G}_m^n} →
  \mathcal{H}¹_{\et}(μ_n)$.  Combining these facts and observing that $B\gm$ is
  $\bA¹$-connected yields the first isomorphism.

  The kernel of the map $\bpi_1^{\bA¹}(B\gm) → \bpi_1^{\bA¹}(BGL_n)$ is $μ_n$.
  Since $\bpi_2^{\bA¹}(BGL_2) \cong \K^{MW}_2$ and $\bpi_i^{\bA¹}(B\gm) = 0$ for
  $i ≥ 2$, the second result also follows from the long exact sequence
  associated with \eqref{eq:ght5} above.
\end{proof}

\begin{rem}
  One can show that $\bpi_2^{\bA¹}(BPGL_2)$ is the pullback of the diagram
  $\bpi_1^{\bA¹}(\bP¹) → \gm \hookleftarrow μ_2$, in particular a subgroup sheaf
  of $\bpi_1^{\bA¹}(\bP¹)$.  The group structure on the above extension is
  inherited from this inclusion.
\end{rem}

\subsection{$PGL_n$-torsors vs.\ $GL_n$-torsors}

If $X$ is a smooth scheme, then there is a function
\[
[X,BPGL_n]_s \longrightarrow [X,BPGL_n]_{\bA¹}
\]
induced by the map sending $BPGL_n$ to its $\bA¹$-localization.  In general,
there is no reason for this function to be surjective, as $\bA¹$-homotopy
classes of maps with source $X$ and target $BPGL_n$ need not come from an actual
$PGL_n$-torsor on $X$.  The next result is a partial replacement for this
deficiency.

\begin{prop}\label{prop:representingpglntorsors}
  Let $n ≥ 2$ be a natural number, and assume that the base field $k$ has
  characteristic that does not divide $n$.  For $X$ a smooth $k$-scheme, the
  canonical map
  \[
  [X,BGL_n]_{\bA¹} → [X,BPGL_n]_{\bA¹}
  \]
  is surjective.  Moreover, given any element $\zeta ∈ [X,BPGL_n]_{\bA¹}$ and
  any smooth affine scheme $Y$ that is $\bA¹$-weakly equivalent to $X$, then
  there exists a vector bundle $\sE$ on $Y$ such that the map $\zeta$ is
  $\bA¹$-homotopic to the classifying map of the $PGL_n$-torsor associated with
  $\sE$.
\end{prop}
\begin{proof}
  We consider the Moore-Postnikov factorization of the map $BGL_n → BPGL_n$.
  For details regarding the Moore-Postnikov factorization in $\bA¹$-homotopy
  theory, we refer the reader to \cite[Thm.~6.1.1]{AsokFaselA3minus0}.  Roughly
  speaking, this factorization corresponds to looking at the Postnikov tower of
  the $\bA¹$-homotopy fiber of $BGL_n → BPGL_n$, which we identified above with
  $B\gm$.  There is a canonical action of $\bpi_1^{\bA¹}(BPGL_n)$ on the
  $\bA¹$-homotopy fiber of $BGL_n → BPGL_n$ induced by change of base-point.
  This yields an action of $\mathcal{H}¹_{\et}(μ_n)$ on $\bpi_i^{\bA¹}(B\gm)$
  and the latter is only non-trivial if $i = 1$, in which case it is isomorphic
  to $\gm$.

  The sheaf of automorphisms of $\gm$ is isomorphic to the constant sheaf
  $\bZ/2$, which is, in particular, strictly $\bA¹$-invariant.  The action of
  $\mathcal{H}¹_{\et}(μ_n)$ on the homotopy sheaves of $B\gm$ is determined by a
  homomorphism of sheaves $\mathcal{H}¹_{\et}(μ_n) → \bZ/2$.  The source and
  target sheaves here are strictly $\bA¹$-invariant and consequently such a
  homomorphism is uniquely determined by its behavior on sections over finitely
  generated extensions of the base-field.  Since $\bZ/2$ is a constant sheaf, to
  determine the value of such a homomorphism over a finitely generated extension
  $L$ of the base field, we can pass to an algebraic closure $\bar{L}$ of $L$.
  In that case, the sections of $\mathcal{H}¹_{\et}(μ_n)$ are necessarily
  trivial, so we conclude that any morphism of sheaves $\mathcal{H}¹_{\et}(μ_n)
  → \bZ/2$ is trivial.

  It follows that there is precisely one obstruction to lifting an
  $\bA¹$-homotopy class of maps $X → BPGL_n$ to an $\bA¹$-homotopy class of maps
  $X → BGL_n$, and that obstruction lies in the group $H²_{\Nis} \bigl(X,\, \gm
  \bigr)$.  Note that we have an untwisted obstruction in this case because the
  action of $\bpi_1^{\bA¹}(BPGL_n)$ on $\gm$ is trivial.  We refer the reader to
  \cite[§~6.1]{AsokFaselA3minus0} for more details on these twisted
  obstructions.

  We now claim that the group $H²_{\Nis} \bigl(X,\, \gm \bigr)$ vanishes for any
  smooth scheme.  Indeed, since the sheaf $\gm$ is strictly $\bA¹$-invariant, we
  know that $H²_{\Nis} \bigl(X,\, \gm \bigr) \cong H²_{\Zar} \bigl(X,\, \gm
  \bigr)$ and the latter cohomology can be computed by means of the Cousin
  resolution for $\gm = \K^M_1$.  This fact is standard, but it is difficult to
  find an explicit reference.  In lieu of a reference of this precise fact, we
  refer the reader to \cite[Props.~5.6--5.8]{quillen} where much more general
  results are established.

  Since $X$ is smooth, the Jouanolou-Thomason homotopy lemma asserts that there
  exists a smooth affine scheme $Y$ that is $\bA¹$-weakly equivalent to $X$,
  \cite[Prop.~4.4]{WeibelHAK}.  Thus, there is a bijection between isomorphism
  classes of rank-$n$ vector bundles on $Y$ and (free) $\bA¹$-homotopy classes
  of maps $[X,BGL_n]$ by \cite[Thm.~8.1.3]{MField}.  In particular, the lift
  constructed in the previous paragraph is represented by a vector bundle on
  $Y$.  It is straightforward to check that the $PGL_n$-torsor associated with
  this vector bundle has the properties mentioned in the statement.
\end{proof}

\subsection{\texorpdfstring{$\bA¹$}{A\textonesuperior}-homotopy classification of $PGL_2$-torsors on \texorpdfstring{$\bP²$}{P\texttwosuperior}}

If $X$ is any smooth variety, then mapping $X_+$ into the $\bA¹$-fiber sequence
$B\gm → BGL_n → BPGL_n$ and using \cite[Prop.~4.3.8]{MV} to identify $[X,B\gm]
\cong \Pic(X)$ yields an exact sequence of groups and pointed sets of the form
\[
[X,PGL_n]_{\bA¹} \longrightarrow \Pic(X) \longrightarrow [X,BGL_n]_{\bA¹} \longrightarrow [X,BPGL_n]_{\bA¹}.
\]
The action of $\Pic(X)$ on $[X,BGL_n]_{\bA¹}$ admits the following description.
While $[X,BGL_n]_{\bA¹}$ need not be in bijection with the set of isomorphism
classes of vector bundles on $X$ if $X$ is not affine, we can always find a
smooth affine scheme $X'$ and an $\bA¹$-weak equivalence $X' → X$.  In that
case, for any space $\mathscr{Y}$, the induced map $[X,\mathscr{Y}]_{\bA¹} →
[X',\mathscr{Y}]_{\bA¹}$ is a bijection.  Thus, we obtain an exact sequence as
above with $X$ replaced by $X'$ throughout.

In that case, we identify $\Pic(X')$ with the set of isomorphism classes of line
bundles on $X$, $[X',BGL_n]_{\bA¹}$ with the set of isomorphism classes of rank
$n$ vector bundles on $X'$ and describe the action of $\Pic(X')$ on
$[X',BGL_n]_{\bA¹}$ as follows,
$$
\Pic(X') ⨯ [X',BGL_n]_{\bA¹} → [X',BGL_n]_{\bA¹} \qquad (\sL, \sE) \mapsto \sL
\otimes \sE.
$$
With these identifications, the next result follows from
Proposition~\ref{prop:representingpglntorsors}.

\begin{prop}\label{prop:homotopyclassificationpart1}
  Let $n ≥ 2$ be a natural number, and assume that the base field $k$ has
  characteristic that does not divide $n$.  For $X$ a smooth $k$-scheme, there
  is a canonical bijection
  \[
  \factor{[X,BGL_n]_{\bA¹}}{\Pic(X)} \isomto [X,BPGL_n]_{\bA¹}, \eqno\qed
  \]
  where the action of $\Pic(X)$ on $[X,BGL_n]_{\bA¹}$ is, up to $\bA¹$-weak
  equivalence described in the preceding paragraph.
\end{prop}

The Chow ring of $BGL_2$ is isomorphic to a formal power series ring over $\bZ$
in two variables $c_1$ and $c_2$, the first and second Chern class.  It follows
from this observation that $c_1$ and $c_2$ yield well-defined (pointed)
functions
\[
c_i: [X,BGL_2]_{\bA¹} \longrightarrow CH^i(X).
\]
These functions are useful in describing the set $[X,BGL_2]_{\bA¹}$ for $X$ a
smooth surface.  More precisely, we have the following result.

\begin{lem}\label{lem:homotopyvectorbundlesonX}
  Assume that the base field $k$ has characteristic unequal to two, and that $X$
  is a (connected) smooth $k$-scheme which is $\bA¹$-weakly equivalent to a
  smooth scheme of dimension $≤ 2$.  Then the map $(c_1,c_2): [X,BGL_2]_{\bA¹} →
  \Pic(X) ⨯ CH²(X)$ is a bijection.
\end{lem}
\begin{proof}
  We compute $[X,BGL_2]_{\bA¹}$ using obstruction theory.  By the same argument
  as \cite[Prop.~6.2]{AsokFaselThreefolds}, under the hypothesis on $X$, the
  canonical map
  \[
  [X,BGL_2^{(2)}]_{\bA¹} \longrightarrow [X,BGL_2]_{\bA¹}
  \]
  is a bijection.

  The second stage of the Postnikov tower of $BGL_2$ is described in
  \cite[§~6]{AsokFaselThreefolds}.  In particular, if $X$ is as in the
  statement, then by \cite[Prop.~6.3]{AsokFaselThreefolds} a map $X →
  BGL_2^{(2)}$ consists of a pair $(\sL,α)$ where $\sL$ is a line bundle on $X$,
  and $α ∈ H² \bigl(X, \, \K^{MW}_2(\sL) \bigr)$.  If $\sL'$ is another line
  bundle on $X$, then there are canonical isomorphisms $H² \bigl(X, \,
  \K^{MW}_2(\sL) \bigr) \cong H² \bigl( X,\, \K^{MW}_2(\sL \tensor
  {\sL'}^{\tensor 2}) \bigr)$.  Since the base field $k$ is assumed
  algebraically closed, the canonical map $H² \bigl( X,\, \K^{MW}_2(\sL) \bigr)
  → H² \bigl(X,\, \K^M_2 \bigr)$ is a bijection, cf.~the proof of
  \cite[Cor.~5.3]{AsokFaselThreefolds}.  In that case, the identification of
  $\sL$ with $c_1$ is clear, and the identification of the class in $H²
  \bigl(\bP²,\,\K^{MW}_2(\sL) \bigr)$ with $c_2$ is contained in the proof of
  \cite[Thm.~6.6]{AsokFaselThreefolds}.
\end{proof}

\begin{thm}\label{thm:pgl2torsors}
  Assume that the base field $k$ has characteristic unequal to two, and that $X$
  is a (connected) smooth $k$-scheme which is $\bA¹$-weakly equivalent to a
  smooth scheme of dimension $≤ 2$.  Then there is a bijection
  \[
  \factor{\Pic(X) ⨯ CH²(X)}{\Pic(X)} \isomto [X,BPGL_2]_{\bA¹},
  \]
  where the action of $l ∈ \Pic(X)$ on $(c_1,c_2) ∈ \Pic(X) ⨯ CH²(X)$ is given
  by the formula
  \[
  l \cdot (c_1,c_2) = (c_1 + 2l,c_2 + lc_1 + l²).
  \]
\end{thm}
\begin{proof}
  This follows from Proposition~\ref{prop:homotopyclassificationpart1},
  Lemma~\ref{lem:homotopyvectorbundlesonX} and the following observation.  If
  $\sL$ is a line bundle on $X$ whose class in $\Pic(X)$ is $l$, and if $\sE$ is
  a rank-two vector bundle with Chern classes $c_1 = c_1(\sE)$ and $c_2 =
  c_2(\sE)$, then $c_1(\sL \tensor \sE) = c_1(\sE) + 2c_1(\sL)$ while $c_2(\sL
  \tensor \sE) = c_2(\sE) + c_1(\sE) c_1(\sL) + c_1(\sL)²$.
\end{proof}

If $H$ is a hyperplane class on $\bP²$, then $\Pic(\bP²) ⨯ CH²(\bP²) \cong \bZ
\cdot H ⨯ \bZ \cdot H²$.  In this special case, Theorem~\ref{thm:pgl2torsors}
simplifies to the following result.

\begin{cor}\label{cor:pgl2torsorsonp2}
  Assume that the base field $k$ has characteristic unequal to two.  Then there
  is an identification
  \[
  (\bZ\cdot H \oplus \bZ \cdot H²)/\bZ \isomto [\bP²,BPGL_2]_{\bA¹},
  \]
  where $\bZ$ acts on $\bZ^{\oplus 2}$ by the formula
  \[
  n \cdot (a,b) = (a + 2n,b + an + n²).  \eqno\qed
  \]
\end{cor}

\begin{rem}
  Schwarzenberger showed that for arbitrary pairs of integers $(a,b)$, there
  exists a vector bundle on $\bP²$ with first Chern class $a$ and second Chern
  class $b$, \cite[Thm.~8]{schwarzenberger:surfaces}.  In particular, his
  construction yields an alternative verification of the surjectivity of
  \[
  H¹_{\Nis} \bigl(\bP²,\,PGL_2 \bigr) \longrightarrow [\bP²,BPGL_2]_{\bA¹}
  \]
  that is independent of Proposition~\ref{prop:representingpglntorsors}.
\end{rem}

%
%

\section{\texorpdfstring{$\bA¹$}{A\textonesuperior}-homotopy classification of
  \texorpdfstring{$\bP¹$}{P\textonesuperior}-bundles over \texorpdfstring{$\bP²$}{P\texttwosuperior}}
\label{sec:04}

In this section, we classify Nisnevich locally trivial $\bP¹$-bundles over
$\bP²$ up to $\bA¹$-weak equivalence of total spaces, using
Corollary~\ref{cor:pgl2torsorsonp2}.  Each Nisnevich locally trivial
$PGL_2$-torsor on a smooth scheme $X$ yields a Nisnevich locally trivial
$\bP¹$-bundle on $X$ by ``passing to the associated fiber space''.  Conversely,
any Nisnevich locally trivial $\bP¹$-bundle on $X$ yields a $PGL_2$-torsor on
$X$ by ``forming the scheme of automorphisms''.  A Nisnevich locally trivial
$\bP¹$-bundle on a smooth scheme is automatically Zariski locally trivial and is
therefore the projectivization of a rank-two vector bundle on $X$.  Thus, this
section aims to classify projectivizations of rank-two vector bundles on $\bP²$
up to $\bA¹$-weak equivalence.

We first show in
Corollary~\ref{cor:weaklyequivalentbundleshaveweaklyequivalenttotalspaces} that,
given a pair of Nisnevich locally trivial $PGL_2$-torsors on a smooth scheme $X$
whose classifying maps coincide in $[X,BPGL_2]_{\bA¹}$, the total spaces of the
associated $\bP¹$-bundles on $X$ are $\bA¹$-weakly equivalent.  Specializing to
the case where $X = \bP²$, we then observe in
Theorem~\ref{thm:a1homotopyp1bundles}, by means of Chow ring computations, that
the $\bP¹$-bundles corresponding to distinct elements of $[\bP²,BPGL_2]_{\bA¹}$
can be distinguished.

\subsection{\texorpdfstring{$\bA¹$}{A\textonesuperior}-classification of projective bundles}

Let $n ≥ 2$ be a natural number, and assume that the characteristic of $k$ does
not divide $n$.  Write $\mathscr{Gr}_n$ for the infinite Grassmannian
parametrizing $n$-dimensional subspaces of the free $k$-vector space generated
by $\bZ$.  Suppose $X$ is a (connected) smooth $k$-scheme, and $\sE$ is a
rank-$n$ vector bundle on $X$.  Since $\mathscr{Gr}_n$ is $\bA¹$-weakly
equivalent to the space $BGL_n$ described in the previous Section~\ref{sec:03},
the classifying map $f_{\sE}$ of $\sE$, as discussed in
Section~\ref{ss:torsors}, determines an element in $[X,\mathscr{Gr}_n]_{\bA¹}$.

Write $γ_n$ for the universal rank-$n$ vector bundle on $\mathscr{Gr}_n$.  The
class of the map $f_{\sE}$ in $[X,\mathscr{Gr}_n]_{\bA¹}$ need not be
represented by an actual morphism from $X$ to $\mathscr{Gr}_n$.  Since $X$ is
smooth, the Jouanolou-Thomason homotopy lemma, \cite[Prop.~4.4]{WeibelHAK},
guarantees that there always exists a smooth affine scheme $X'$ and a morphism
$π: X' → X$ that is a torsor under a vector bundle on $X$.  In particular, $π$
is an $\bA¹$-weak equivalence.  In general, the pair $(X',π)$ is not unique, and
we refer to a choice of such a pair as \emph{a Jouanolou device}.  If we write
$\sE'$ for $π^* \sE$, then by \cite[Thm.~8.1.3]{MField} the classifying map
$f_{\sE'}$ is represented by a morphism $X → \mathscr{Gr}_n$ that, by abuse of
terminology, we will also denote $f_{\sE'}: X' → \mathscr{Gr}_n$.

It follows that the morphism $\bP_{X'}(\sE') → X'$ is the pullback of
$\bP_{\mathscr{Gr}_n}(γ_n) → \mathscr{Gr}_n$ along the morphism $f_{\sE'}$ of
the previous paragraph.  On the other hand, there is a pullback square of the
form
\[
\xymatrix{
  \bP_{X'}(\sE') \ar[r]\ar[d]& X' \ar[d] \\
  \bP_{X}(\sE) \ar[r] & X }
\]
since $\bP_{X}(\sE) ⨯_{X} X' \cong \bP_{X'}(\sE')$ by \cite[Prop.~3.5.3]{EGA2}.
Since the right hand vertical morphism is a torsor under a vector bundle, the
left hand vertical map is a torsor under a vector bundle as well\footnote{In
  fact, it is a torsor under the pull-back along $π$ of the vector bundle on $X$
  under which $π$ is a torsor.} and is, in particular, an $\bA¹$-weak
equivalence.

\begin{prop}\label{prop:vectorbundlesgiveweakequivalences}
  Let $n≥ 2$ be a natural number, and assume $k$ has characteristic that does
  not divide $n$.  Suppose $X$ is a smooth $k$-scheme and $\sE_0$, $\sE_1$ are a
  pair of rank-$n$ vector bundles on $X$ with classifying maps $f_0$ and $f_1$.
  If the classes of $f_0$ and $f_1$ are equal in $[X,\Grass]_{\bA¹}$, then the
  projective bundles $\bP_X(\sE_0)$ and $\bP_X(\sE_1)$ are $\bA¹$-weakly
  equivalent.
\end{prop}
\begin{proof}
  We produce an explicit chain of three $\bA¹$-weak equivalences between the two
  projective bundles.  First, fix a Jouanolou device $π: X' → X$.  Write $\sE_i'
  := π^* \sE_i$.  By the discussion just prior to the statement, the maps
  $\bP_{X'}(\sE'_i) → \bP_{X}(\sE_i)$ are $\bA¹$-weak equivalences.  If the
  classifying maps $f_i$ of the vector bundles $\sE_i$ lie in the same class in
  $[X,\mathscr{Gr}_n]_{\bA¹}$, then, since the map $[X,\mathscr{Gr}_n]_{\bA¹} →
  [X',\mathscr{Gr}_n]_{\bA¹}$ induced by pullback is a bijection, it follows
  from \cite[Thm.~8.1.3]{MField} that the bundles $\sE'_i$ are actually
  isomorphic as vector bundles on $X'$.  A choice of such an isomorphism induces
  an isomorphism of the total spaces of the associated projective bundles
  $\bP_{X'}(\sE'_0) \cong \bP_{X'}(\sE'_1)$.  Thus we have constructed a diagram
  $$
  \bP_{X}(\sE_0) \longleftarrow \bP_{X'}(\sE'_0) \longrightarrow
  \bP_{X'}(\sE'_1) \longrightarrow \bP_{X}(\sE_1)
  $$
  where each morphism is an $\bA¹$-weak equivalence.
\end{proof}

\begin{cor}\label{cor:weaklyequivalentbundleshaveweaklyequivalenttotalspaces}
  Assume that $k$ has characteristic unequal to two and that $X$ is a smooth
  $k$-scheme that is $\bA¹$-weakly equivalent to a smooth $k$-scheme of
  dimension $≤ 2$.  Suppose $\sE_0$, $\sE_1$ are two rank-two vector bundles on
  $X$.  If $f_i$ is the classifying morphism of the Zariski locally trivial
  $PGL_2$-torsor associated with $\bP_{X}(\sE_i)$, and if the classes of $f_i$
  in $[X,BPGL_2]_{\bA¹}$ coincide, then $\bP_{X}(\sE_0)$ is $\bA¹$-weakly
  equivalent to $\bP_{X}(\sE_1)$.
\end{cor}
\begin{proof}
  The classifying maps $f_i$ lie in the same class in $[X,BPGL_2]_{\bA¹}$ by
  assumption.  By Theorem~\ref{thm:pgl2torsors} it follows that the Chern
  classes of $\sE_0$ and those of $\sE_1$ lie in the same orbit for the action
  of $\Pic(X)$ on $\Pic(X) ⨯ CH²(X)$ coming from tensoring by a line
  bundle.

  In other words, there exists $\sL ∈ \Pic(X)$ such that the Chern classes of
  the twist $\sE_0 \tensor \sL$ coincide with those of $\sE_1$.  Therefore, by
  Lemma~\ref{lem:homotopyvectorbundlesonX}, the classifying maps of the vector
  bundles $\sE_0 \tensor \sL$ and $\sE_1$ lie in the same class in
  $[X,\mathscr{Gr}_n]_{\bA¹}$.  Applying
  Proposition~\ref{prop:vectorbundlesgiveweakequivalences} in this situation
  allows us to complete the proof.
\end{proof}

\subsection{Chow rings and Chern classes}
\CounterStep

\begin{computation}
\label{comp:chow}
  If $X$ is a smooth scheme, if $\sE$ is a rank-$n$ vector bundle on $X$ and
  $\bP_X(\sE)$ is the associated projective space bundle, then the Chow ring
  $CH^* \bigl( \bP_X(\sE) \bigr)$ is described by the projective bundle formula.
  More precisely,
  \begin{equation}\label{eq:CP}
    CH^* \bigl( \bP_X(\sE) \bigr) \cong \factor{CH^*(X)[τ]}{ \bigl\langle
      P_{\sE}(τ) \bigr\rangle}, \quad \text{where} \quad P_{\sE}(τ) :=
    \sum_{i=0}^n c_i(\sE)τ^{n-i}.
  \end{equation}
  If $X = \bP²$, if $H$ is a hyperplane class, and $\sE$ a rank-two vector
  bundle on $\bP²$ with Chern classes $c_1$, $c_2$, then \eqref{eq:CP}
  simplifies to
  $$
  CH^* \bigl(\bP_{\bP²}(\sE) \bigr) \cong \factor{\bZ[H,τ]}{\langle H^3,τ² + c_1
    Hτ + c_2 H² \rangle}.
  $$
  The ring structure equips $\Pic \bigl(\bP_{\bP²}(\sE) \bigr)$ with an integral
  cubic form which is computed to be the following,
  $$
  Φ : \Pic \bigl( \bP_{\bP²}(\sE) \bigr) → CH^3 \bigl( \bP_{\bP²}(\sE)
  \bigr) \cong \bZ, \qquad aH + bτ \mapsto 3a²b-3c_1ab²+(c_1²-c_2)b^3.
  $$
  The discriminant of $Φ$ is $c_1² - 4c_2$.
\end{computation}

Now assume that we are given two rank-two bundles on $\bP²$, say $\sE_1$ and
$\sE_2$, with arbitrary Chern classes.  Any isomorphism of graded rings, $CH^*
\bigl( \bP_{\bP²}(\sE_1) \bigr) → CH^* \bigl( \bP_{\bP²}(\sE_2) \bigr)$, induces
an invertible linear map of Picard groups, $\Pic \bigl(\bP_{\bP²}(\sE_1) \bigr)
→ \Pic \bigl(\bP_{\bP²}(\sE_2) \bigr)$, which, in terms of the bases above can
be identified with an element of $GL_2(\bZ)$.  From \cite[Ex.~5,
Prop.~18]{okonek:vandeven}, it follows that the $GL_2(\bZ)$-orbits are
distinguished by the discriminant.  We formulate this as a lemma.

\begin{lem}\label{lem:chowring}
  Let $\sE_1$ and $\sE_2$ be any two rank-two vector bundles on $\bP²$.  Then,
  the Chow rings of $\bP(\sE_1)$ and $\bP(\sE_2)$ are isomorphic if and only if
  the discriminants of the associated cubic forms on Picard groups are
  equal.  \qed
\end{lem}

\begin{thm}\label{thm:a1homotopyp1bundles}
  Let $k$ be an algebraically closed field having characteristic unequal to
  two.  Suppose $\sE$ and $\sE'$ are rank-two vector bundles on $\bP²$ with Chern
  classes $(c_1,c_2)$ and $(c_1',c_2')$, respectively.  Then, an $\bA¹$-weak
  equivalence $\bP_{\bP²}(\sE) \aoequiv \bP_{\bP²}(\sE')$ exists if and only if
  $(c_1,c_2)$ and $(c_1',c_2')$ lie in the same orbit for the $\bZ$-action on
  $\bZ \cdot H \oplus \bZ \cdot H²$ described in
  Corollary~\ref{cor:pgl2torsorsonp2}.
\end{thm}
\begin{proof}
  If $(c_1,c_2)$ lies in the same $\bZ$-orbit as $(c_1',c_2')$, then the
  associated projective bundles are $\bA¹$-weakly equivalent: by
  Corollary~\ref{cor:pgl2torsorsonp2} the $\bA¹$-homotopy class of
  $[\bP²,BPGL_2]$ is equivalent to specifying the $\bZ$-orbit of the pair
  $(c_1,c_2)$.  It follows that the $\bA¹$-homotopy classes corresponding to the
  classifying maps of $\bP(\sE)$ and $\bP(\sE')$ agree.  Given this fact,
  Corollary~\ref{cor:weaklyequivalentbundleshaveweaklyequivalenttotalspaces}
  implies that the projective bundles associated with these vector bundles are
  $\bA¹$-weakly equivalent.

  Conversely, suppose that we have an $\bA¹$-weak equivalence $\bP_{\bP²}(\sE)
  \aoequiv \bP_{\bP²}(\sE')$.  In this case, there is a ring isomorphism $CH^*
  \bigl( \bP_{\bP²}(\sE) \bigr) \cong CH^* \bigl( \bP_{\bP²}(\sE') \bigr)$ and
  in particular, the cubic forms on Picard groups are isomorphic and therefore
  have equal discriminants by Lemma~\ref{lem:chowring}.  We need to show that
  $(c_1,c_2)$ and $(c_1',c_2')$ lie in the same $\bZ$-orbit.  Note that
  $c_1(\sE)² - 4c_2(\sE) \equiv c_1(\sE) \mod 2$.  By definition, tensoring
  $\sE$ by a line bundle preserves the $\bZ$-orbit of $(c_1,c_2)$.  After
  replacing $\sE$ and $\sE'$ by $\sE \tensor \sL$ and $\sE' \tensor \sL'$ if
  necessary, we can assume that $c_1(\sE)$ and $c_1(\sE')$ are either both equal
  to zero or both equal to one.  Now, the equality of discriminants implies that
  $4 \cdot c_2(\sE) = 4 \cdot c_2(\sE')$, and therefore the second Chern classes
  of the bundles must be equal as well.  It follows that the vector bundles
  $\sE$ and $\sE'$ can be assumed to have $(c_1,c_2) = (c_1',c_2')$.  But then
  $(c_1,c_2)$ and $(c_1',c_2')$ obviously lie in the same $\bZ$-orbit.
\end{proof}

%
%

\section{Concordance classification of rank-two vector bundles over \texorpdfstring{$\bP²$}{P\texttwosuperior}}
\label{sec:05}

This section discusses the $\bA¹$-concordance classification of rank-two vector
bundles on $\bP²$.  Our main result, Theorem~\ref{thm:concordanceclass}, asserts
that among all bundles with vanishing second Chern class, the first Chern class
is the only concordance-invariant.  To be more precise, we show that any
rank-two vector bundle on $\bP²$ with arbitrary first Chern class $c_1$ and
second Chern class $c_2 = 0$ is $\bA¹$-concordant to the split bundle
$\sO_{\bP²} \oplus \sO_{\bP²}(c_1)$.  This result allows us, in the subsequent
Section~\ref{sec:06}, to obtain $\bA¹$-$h$-cobordism classification results for
projectivizations of ``topologically split'' bundles.

\begin{defn}[Concordance and direct concordance]
  Given a $k$-scheme $X$ and two vector bundles $\sE_0$, $\sE_1$ on $X$, we say
  that $\sE_0$ and $\sE_1$ are \emph{directly $\bA¹$-concordant} if there exists
  a vector bundle $\sE$ over $X ⨯ \bA¹$ such that $\sE_0 \cong ι_0^* \sE$ and
  $\sE_1 \cong ι_1^* \sE$, where $ι_i : X → X ⨯ \{i \} ⊂ X ⨯ \bA¹$ are the
  obvious inclusions.  The vector bundles $\sE_0$ and $\sE_1$ on $X$ are said to
  be \emph{$\bA¹$-concordant} if they are equivalent under the equivalence
  relation generated by direct $\bA¹$-concordance.
\end{defn}

\begin{rem}
  A direct $\bA¹$-concordance is a deformation of vector bundles.  If two vector
  bundles are $\bA¹$-concordant, then they can be deformed into each other, over
  a base space that need not be irreducible.  On the other hand, if two vector
  bundles are deformation equivalent, then they need not be $\bA¹$-concordant
  since the parameter space of the deformation need not contain any affine
  lines.
\end{rem}

\begin{rem}
  The homotopy invariance results of Quillen-Suslin \cite{quillen} and Lindel
  \cite{lindel} imply that for smooth, affine $X$, the notion of
  $\bA¹$-concordance agrees with vector bundle isomorphism.  However, over
  non-smooth or non-affine base schemes, there are non-trivial deformations and
  $\bA¹$-concordances of vector bundles.
\end{rem}

\begin{thm}[Concordance classification of vector bundles with $c_2=0$]\label{thm:concordanceclass}
  If $\sE$ is a rank-two vector bundle on $\bP²$ with arbitrary first Chern
  class $c_1$ and with vanishing second Chern class $c_2=0$, then $\sE$ is
  $\bA¹$-concordant to $\sO_{\bP²} \oplus \sO_{\bP²}(c_1)$.  Thus, any two
  rank-two vector bundles on $\bP²$ with first Chern class $c_1$ and second
  Chern class $0$ are $\bA¹$-concordant.
\end{thm}

A proof of Theorem~\ref{thm:concordanceclass} is given in
Section~\vref{sec:pf:thm:concordanceclass}.

\subsection{Explicit construction of \texorpdfstring{$\bA¹$}{A\textonesuperior}-concordances}
\label{ssec:explicitC}

To prepare for the proof Theorem~\ref{thm:concordanceclass}, we aim to refine
the results of Proposition~\ref{prop:explicit1-x} to statements about
$\bA¹$-concordance.  The proof of Proposition~\ref{prop:explicit1-x} made use of
the irreducibility of the moduli spaces $M(d)$, as well as deformations that
connect bundles in $M(d;e) ⊆ M(d)$ to bundles in $M(e)$.  The deformations that
go from $M(d;e)$ to $M(e)$ are explicitly described in Strømme's paper,
\cite[Sect.~4]{stroemme}, and are easily seen to be $\bA¹$-concordances.  For
the reader's convenience, we briefly recall their construction.

\begin{construction}\label{cons:ex1}
  Fix the following data.
  \begin{enumerate}
  \item A rank-two vector bundle $\sF$ on $\bP²$ with Chern classes $c_1, c_2 ∈
    \bZ$ and splitting type $e ≥ -1$.
  \item An integer $d>e$ and a global section $τ ∈ H^0\bigl( \bP²,\, \sF(d-c_1)
    \bigr)$ that vanishes in a codimension-two subscheme $Y ⊆ \bP²$.
  \item A section $F ∈ H^0\bigl(\bP²,\, \sO(2d-c_1)\bigr)$ whose zero locus is
    disjoint from $Y$.
  \end{enumerate}
  We obtain a sequence
  $$
  0 → \sO_{\bP²} \xrightarrow{τ} \sF(d-c_1) → \sI_Y \otimes \sO_{\bP²}(2d-c_1) → 0
  $$
  and an associated extension class $\xi ∈ \Ext¹ \bigl(\sI_Y,\,
  \sO_{\bP²}(c_1-2d)\bigr)$.

  To continue with the construction, consider the standard projection $π :
  \bP²⨯\bA¹ → \bP²$, consider the map
  \[
  τ \wedge (-) : \sF → \underbrace{(\wedge² \sF)(d-c_1)}_{\cong \sO_{\bP²}(d)}
  \qquad σ \mapsto τ \wedge σ,
  \]
  and choose a coordinate $T$ on $\bA¹$.  We obtain the following monad of
  vector bundles on $\bP²⨯\bA¹$,
  \[
  π^* \sO(c_1-d)\xrightarrow{b := (T \cdot \Id,τ,-F)^t} π^* \sO(c_1-d) \oplus
  π^* \sF \oplus π^* \sO(d) \xrightarrow{a := (F,τ\wedge(-),T \cdot \Id)} π^*
  \sO(d).
  \]
  The vanishing loci of $τ$ and $-F$ are disjoint by assumption.  This implies
  that $b$ is injective and that $a$ is surjective.  It also implies that both
  maps have constant rank, so that the cohomology sheaf $\sC$ is locally free.
  Given $i∈\{0,1\}$, let $ι_i : \bP²⨯\{i\} → \bP²⨯\bA¹$ be the corresponding
  inclusion.  The bundle $\sC$ provides an $\bA¹$-concordance from the bundle
  $\sF_0 \cong ι_0^*(\sC)$ to the bundle $\sF_1 := ι_1^*(\sC)$.  Strømme proves
  in \cite[Prop.~4.3]{stroemme} that the following properties hold.
  \begin{enumerate}
  \item The bundle $\sF_1$ is isomorphic to $\sF$.
  \item The bundle $\sF_0$ has splitting type $d$.  The Chern classes of $\sF_0$
    and $\sF_1$ agree.
  \item\label{il:w3} The bundle $\sF_0$ appears in a sequence
    $$
    0 → \sO_{\bP²} → \sF_0(-d) → \sI_Y \otimes \sO_{\bP²}(c_1-2d) → 0.
    $$
    whose extension class is the image of $\xi$ under the map
    $$
    F²\cdot(-): \Ext¹ \bigl( \sI_Y,\, \sO_{\bP²}(c_1-2d) \bigr) → \Ext¹
    \bigl(\sI_Y,\, \sO_{\bP²}(2d-c_1) \bigr).
    $$
  \end{enumerate}
\end{construction}

\begin{rem}[Explanation of $F²\cdot(-)$]\label{rem:discextcls}
  Using the Hartshorne-Serre correspondence, we obtain the following diagram,
  which gives an elementary description of the map $F²\cdot(-)$ that appears in
  Item~\ref{il:w3} above,
  $$
  \xymatrix{ %
    \Ext¹ \bigl( \sI_Y,\, \sO_{\bP²}(c_1-2d) \bigr) \ar@{^(->}[d]_{\text{Hartshorne-Serre, Fact~\ref{fact:arrondo1}}} \ar[r]^{F²\cdot(-)} & \Ext¹ \bigl(\sI_Y,\,\sO_{\bP²}(2d-c_1) \bigr) \ar@{^(->}[d]^{\text{Hartshorne-Serre, Fact~\ref{fact:arrondo1}}}\\
    H^0 \bigl( Y,\, \wedge² \sN_{X/Y}\otimes \sO_{\bP²}(c_1-2d)|_Y \bigr)
    \ar[r]_{F²\cdot(-)} & H^0 \bigl( Y,\, \wedge² \sN_{X/Y}\otimes
    \sO_{\bP²}(2d-c_1)|_Y \bigr).  }
  $$
  The vertical arrows are injective since $h¹ \bigl( X,\, \sO_{\bP²}(c_1-2d)
  \bigr) = h¹ \bigl( X,\, \sO_{\bP²}(2d-c_1) \bigr) = 0$.
\end{rem}

\subsection{Proof of Theorem~\ref*{thm:concordanceclass}}
\label{sec:pf:thm:concordanceclass}

Let $\sE$ be a rank-two vector bundle on $\bP²$ with arbitrary $c_1$ and
vanishing $c_2=0$, as in the formulation of Theorem~\ref{thm:concordanceclass}.
We aim to deform $\sE$ to the split bundle $\sO_{\bP²}(c_1) \oplus \sO_{\bP²}$.
To this end we produce three chains of $\bA¹$-concordances that together give
the desired deformation.

\subsubsection*{Step 1: Increasing the splitting type}

First, we fix an integer $N \gg 0$ satisfying the following three conditions:
the sheaf $\sE(N-1)$ is globally generated, $N > c_1$ and $2N - c_1 > 0$.  We
deform $\sE$ to a bundle with the same Chern classes and splitting type $N$.  By
the Bertini theorem \ref{fact:bertini}, the first hypothesis on $N$ guarantees
that the vanishing locus of a generic section of $\sE(N)$ is smooth; fix a
section $τ ∈ H^0 \bigl( \bP²,\,\sE(N) \bigr)$ whose vanishing locus $Y$ is a
zero-dimensional, smooth scheme.  Applying Construction~\ref{cons:ex1}, we
obtain an $\bA¹$-concordance from $\sE$ to a bundle $\sE_0$ with Chern classes
$c_1$ and $0$ and splitting type $N$.  Sequence~\ref{il:w3} then yields that
$$
0 = c_2(\sE_0) = c_1 \bigl( \sO_{\bP²}(N) \bigr) \cdot c_1 \bigl(
\sO_{\bP²}(c_1-N) \bigr) + c_2(\sI_Y) \quad \Rightarrow \quad \#Y = N \cdot
(N-c_1).
$$

Replacing $\sE$ by $\sE_0$, we are free to make the following assumption for the
remainder of the present proof.

\begin{asswlog}\label{asswlog:x1}
  If $N$ denotes the splitting type of $\sE$, then $N > c_1$ and $2N - c_1 > 0$.
  Further, there exists a reduced, finite subscheme $Y ⊂ \bP²$ of length $N
  \cdot (N-c_1)$, a section $s ∈ H^0 \bigl( \bP², \, \sE(-N) \bigr)$ vanishing
  precisely on $Y$ and giving rise to an exact sequence
  \begin{equation}\label{eq:sdcvb}
    0 → \sO_{\bP²} \xrightarrow{s} \sE(-N) → \sI_Y \otimes \sO_{\bP²}(c_1-2\cdot N) → 0.
  \end{equation}
\end{asswlog}

\subsubsection*{Step 2: Deforming the subscheme}

Next, we aim to deform the bundle $\sE$ to a new bundle with the same Chern
classes and splitting type, but for which the associated zero-dimensional
subscheme is the intersection of two curves of appropriate degree.  To this end,
we apply the extension result for vector bundles, Corollary~\ref{cor:rhsc},
which we obtained as a corollary to the Hartshorne-Serre correspondence.  More
precisely, fix a pair of smooth curves $C$ and $D$ of degrees $N$ and $N-c_1$,
respectively, intersecting in a reduced subscheme $Y_1$ consisting of $N(N-c_1)$
distinct points.  Write $X = \bP² ⨯ \bA¹$ and choose a subvariety $Y_X ⊂ X$ that
is the union of $N(N-c_1)$ pairwise disjoint sections, with $Y_X|_{\bP²⨯\{0\}} =
Y$ and $Y_X|_{\bP²⨯\{1\}} = Y_1$.  Corollary~\ref{cor:rhsc} will then allow to
find an $\bA¹$-concordance between $\sE$ and a bundle $\sE_1$ that has the
splitting type and Chern classes of $\sE$, and fits into an exact sequence,
\begin{equation}\label{eq:h788}
  0 → \sO_{\bP²} \xrightarrow{s} \sE_1(-N) → \sI_{Y_1} \otimes \sO_{\bP²}(c_1-2\cdot N) → 0.
\end{equation}
Sequence~\eqref{eq:h788} immediately implies that the splitting type of $\sE_1$
is $N$.  Replacing $\sE$ by $\sE_1$, we are free to make the following
assumption for the remainder of the present proof.

\begin{asswlog}\label{asswlog:x2}
  In addition to the assumptions made in \ref{asswlog:x1}, we can further assume
  the reduced scheme $Y$ is the intersection of two smooth curves, say
  $C$ and $D$, of degrees $N$ and $N-c_1$, respectively.
\end{asswlog}

The concluding Step~3 of this proof discusses the bundle $\sE$ in the context of
Construction~\ref{cons:ex1}.  To fix the necessary notation, we briefly discuss
the extension class of Sequence~\eqref{eq:sdcvb}, which determines the
isomorphism class of $\sE$.  The assumptions on the integer $N$ guarantee that
the following cohomology groups vanish:
\begin{equation}\label{eq:birne}
  H¹ \bigl( X,\, \sO_{\bP²}(2N-c_1) \bigr) = 0 \quad\text{and}\quad H² \bigl( X,\, \sO_{\bP²}(2N-c_1) \bigr) = 0.
\end{equation}
Using the Hartshorne-Serre correspondence of Fact~\ref{fact:arrondo1} and
Remark~\ref{ram:arrondo2}, the splitting type of Sequence~\eqref{eq:sdcvb} is
thus identified with an element
\begin{equation}\label{eq:apfel}
  \xi_{\sE} ∈ H^0 \bigl( \bP²,\, \wedge² \sN_{X/Y}\otimes \sO_{\bP²}(2N-c_1)|_Y \bigr) \cong \Ext¹ \bigl( \sI_Y,\, \sO_{\bP²}(2N-c_1) \bigr)
\end{equation}
that will later become important.

\subsubsection*{Step 3: Deforming to a split bundle}

Generalizing \cite[Rem.~4.6]{stroemme}, we show that $\sE$ can be deformed to
the split bundle $\sO_{\bP²} \oplus \sO_{\bP²}(c_1)$.  We begin by showing that
the restriction map $r: H^0 \bigl( \bP², \sO_{\bP²}(2N-c_1) \bigr) → H^0 \bigl(
Y, \sO_{\bP²}(2N-c_1)|_Y \bigr)$ is surjective.  To this end, factor $r$ as
follows:
$$
H^0 \bigl( \bP²,\, \sO_{\bP²}(2N-c_1) \bigr) \xrightarrow{r_1} H^0 \bigl( C,\,
\sO_{\bP²}(2N-c_1)|_C \bigr) \xrightarrow{r_2} H^0 \bigl( Y,\,
\sO_{\bP²}(2N-c_1)|_Y \bigr).
$$
Surjectivity of $r_1$ follows by noting that its cokernel is controlled by $H¹
\bigl( \bP², \sO_{\bP²}(N-c_1) \bigr) = 0$.  Surjectivity of $r_2$ follows by
noting that
\begin{align*}
  H¹ \bigl( C,\, \sO_{\bP²}(2N-c_1)|_C \otimes \sI_{Y} \bigr) & \cong H¹ \bigl( C,\, \sO_C((2N-c_1)N-N(N-c_1)) \bigr) \\
  & = H¹ \bigl( C,\, \sO_C(N²) \bigr) \\
  & \cong H^0 \bigl( C,\, ω_C \otimes \sO_C(-N²) \bigr)^* \\
  & \cong H^0 \bigl( C,\, \sO_C(-3N+N²-N²) \bigr)^* \\
  & = H^0 \bigl( C,\, \sO_C(-3N) \bigr)^* = 0.
\end{align*}
We now appeal to Construction~\ref{cons:ex1} to produce an $\bA¹$-concordance
between $\sE$ and the bundle $\sF = \sO_{\bP²}(c_1)\oplus \sO_{\bP²}$.  To this
end, set $d := N$, and choose a section
\begin{multline*}
  τ ∈ H^0\bigl( \bP², \sF(N-c_1) \bigr) = H^0\bigl( \bP², \sO_{\bP²}(N)\oplus\sO_{\bP²}(N-c_1)\bigr) \\
  = H^0\bigl( \bP², \sO_{\bP²}(N)\bigr) \oplus H^0\bigl( \bP², \sO_{\bP²}(N-c_1)\bigr)
\end{multline*}
associated with the pair of curves $(C,D)$.  The section $τ$ vanishes precisely
on $Y$ and gives rise to an exact sequence
\begin{equation}\label{eq:gfkh54}
  0 → \sO_{\bP²} \xrightarrow{τ} \sF(N-c_1) → \sI_Y \otimes \sO_{\bP²}(2N-c_1) → 0.
\end{equation}
Using the Hartshorne-Serre correspondence, Fact~\ref{fact:arrondo1}, the
extension class $\zeta_{\sF}$ associated with \eqref{eq:gfkh54} yields an
element
$$
\xi_{\sF} ∈ H^0 \bigl( \bP²,\, \wedge² \sN_{X/Y}\otimes \sO_{\bP²}(c_1-2N)|_Y \bigr).
$$
The characterization of locally frees given in Theorem~\ref{thm:arrondo2}
asserts that both $\xi_{\sE}$ and $\xi_{\sF}$ are nowhere-vanishing on $Y$.
Since $Y$ is finite, we can thus find a section $F_Y ∈ H^0 \bigl( Y,
\sO_{\bP²}(2N-c_1)|_Y \bigr)$ such that
\begin{equation}\label{eq:hj7}
  \xi_{\sE} = F_Y²\cdot \xi_{\sF}.
\end{equation}
Surjectivity of the restriction map $r$ allows us to extend $F_Y$ to a section
$F ∈ H^0 \bigl( \bP², \sO_{\bP²}(2N-c_1) \bigr)$ with vanishing locus disjoint
from $Y$.  Feeding these data into Construction~\ref{cons:ex1}, we obtain an
$\bA¹$-concordance between $\sF$ and a bundle $\sF_0$ that appears in a
sequence
\begin{equation}\label{eq:fgfgh8}
  0 → \sO_{\bP²} → \sF_0(-N) → \sI_Y \otimes \sO_{\bP²}(c_1-2N) → 0,
\end{equation}
and whose extension class is $F² \cdot \zeta_{\sF}$.  To identify this class,
recall the diagram of Remark~\ref{rem:discextcls}, which reads in our context as
follows:
$$
\xymatrix{ %
  \Ext¹ \bigl( \sI_Y,\, \sO_{\bP²}(c_1-2N) \bigr) \ar@{^(->}[d] \ar[rr]^{F²\cdot(-)} && \Ext¹ \bigl(\sI_Y,\,\sO_{\bP²}(2N-c_1) \bigr) \ar@{<->}[d]^{\text{isom.~by \eqref{eq:birne}, \eqref{eq:apfel}}}\\
  H^0 \bigl( Y,\, \wedge² \sN_{X/Y}\otimes \sO_{\bP²}(c_1-2N)|_Y \bigr) \ar[rr]_{F²\cdot(-)} && H^0 \bigl( Y,\, \wedge² \sN_{X/Y}\otimes \sO_{\bP²}(2N-c_1)|_Y \bigr).
}
$$
The upper horizontal morphism maps the extension class $\zeta_{\sF}$ of
\eqref{eq:gfkh54} to the extension class of \eqref{eq:fgfgh8}.  On the other
hand, it follows by construction of $F$ that the lower horizontal morphism maps
$\xi_{\sF}$ to $\xi_{\sE}$, the latter being induced by
Sequence~\eqref{eq:sdcvb}.  Since the vertical arrow on the right is an
isomorphism, we obtain that the extension classes of \eqref{eq:fgfgh8} and
\eqref{eq:sdcvb} agree.  In summary, we have seen that $\sF_0 \cong \sE$, which
concludes the proof of Theorem~\ref{thm:concordanceclass}.  \qed

\subsection{Concluding remarks}

We want to make a remark on the geometry of the moduli spaces $M(d)$ and why
they fail to be $\bA¹$-chain connected.  As in \cite{banica}, the space $M(d)$
is fibered over the Hilbert scheme of local complete intersections of
codimension two in $\bP²$.  The fibers are complements of hyperplane
arrangements in projective spaces $\bP H^0(Y,\sO_Y)$, typically isomorphic to
products of $\bG_m$.  It is possible to produce explicit deformations of lci
subschemes, cf.~the proof of Theorem~\ref{thm:concordanceclass}, to show that
the Hilbert scheme of lci subschemes of codimension two in $\bP²$ is
$\bA¹$-chain connected.  However, the fibers are generally not
$\bA¹$-chain-connected.  Any morphism $\bA¹ → M(d)$ for $d≥ 2$ is a deformation
of the underlying lci subscheme equipped with a constant section.  Vector
bundles differing only in the extension class and not in the subscheme cannot be
connected by an $\bA¹$-chain inside $M(d)$.

It is this subtle geometry of the moduli spaces $M(d)$ that prevents us from
giving a complete concordance classification of rank-two bundles.  While we do
not expect the moduli spaces $M(d)$ to be $\bA¹$-chain connected, there might
exist chains of $\bA¹$-concordances through higher splitting types which connect
bundles that are not $\bA¹$-concordant through bundles of type $d$.  We were not
able to settle this question except in the case of ``topologically split''
bundles presented here.

%
%

\section{\texorpdfstring{$\bA¹$-$h$}{A\textonesuperior-h}-cobordism classification of \texorpdfstring{$\bP¹$}{P\textonesuperior}-bundles over \texorpdfstring{$\bP²$}{P\texttwosuperior}}
\label{sec:06}

In this section, we discuss the classification of $\bP¹$-bundles over $\bP²$ up
to $\bA¹$-$h$-cobordism.  On the one hand, the non-deformability results of
Strømme allow to show that there exist many $\bP¹$-bundles which can not be
connected by direct $\bA¹$-$h$-cobordisms.  On the other hand, the explicit
$\bA¹$-concordances produced in Section~\ref{sec:05} allow us to establish that
any $\bP¹$-bundle deformable to the projectivization of a split vector bundle is
in fact already $\bA¹$-$h$-cobordant to the split bundle.  This provides an
$\bA¹$-$h$-cobordism classification for certain $\bP¹$-bundles, and also
exhibits how far direct $\bA¹$-$h$-cobordism is from being an equivalence
relation.

\subsection{Preliminaries on \texorpdfstring{$\bA¹$-$h$}{A\textonesuperior-h}-cobordisms}

For the reader's convenience, we briefly recall the definition of an
$\bA¹$-$h$-cobordism.

\begin{defn}[\protect{$\bA¹$-$h$-cobordism, \cite[Def.~3.1.1]{AM}}]\label{def:a1hc}
  Given two smooth, proper abstract varieties $X_0$ and $X_1$, an
  $\bA¹$-$h$-cobordism between $X_0$ and $X_1$ is a proper, surjective morphism
  of smooth abstract varieties, $f: \hcob{X} → \bA¹$, such that the following
  holds.
  \begin{enumerate}
  \item The fibers $f^{-1}(0)$ and $f^{-1}(1)$ are isomorphic to $X_0$ and
    $X_1$, respectively.
  \item\label{il:a1we} The natural closed immersions $i_0: X_0 → \hcob{X}$ and
    $i_1: X_1 → \hcob{X}$ are $\bA¹$-weak equivalences.
  \end{enumerate}
  We say that $X_0$ and $X_1$ are directly $\bA¹$-$h$-cobordant if there exists
  an $\bA¹$-$h$-cobordism between $X_0$ and $X_1$.  We say that $X_0$ and $X_1$
  are $\bA¹$-$h$-cobordant if they are equivalent under the equivalence relation
  generated by direct $\bA¹$-$h$-cobordisms.
\end{defn}

\begin{rem}[Smoothness of the cobordism map]\label{rem:a1hcr1}
  In the setting of Definition~\ref{def:a1hc}, we conclude as in
  Remark~\ref{rem:2} that $f$ is smooth over a Zariski-open neighborhood $V ⊆
  \bA¹$ that contains $0$ and $1$.
\end{rem}

\begin{rem}[Restrictions of Picard groups]
  Maintaining the assumptions of Remark~\ref{rem:a1hcr1}, write $\hcob{X}_V :=
  f^{-1}(V)$.  We obtain a commutative diagram of groups,
  $$
  \xymatrix{%
    \bigl[ \hcob{X},\, B\bG_m \bigr]_{\bA¹} \ar@{<->}[r]^\cong \ar@{<->}[d]_{\cong} & \bigl[ X_0,\, B\bG_m \bigr]_{\bA¹} \ar@{<->}[d]^\cong\\
    \Pic(\hcob{X}) \ar[r]_{\text{restriction}} & \Pic(X_0).
  }
  $$
  The isomorphism on top follows from Assumption~\ref{il:a1we}.  The left and
  right bijections follow from smoothness of $\hcob{X}$ and $X_0$, respectively.
  In particular, the natural restriction map $\Pic(\hcob{X}) → \Pic(\hcob{X}_V)
  → \Pic(X_0)$ is bijective, and the restriction $\Pic(\hcob{X}_V)→ \Pic(X_0)$
  is surjective.  The same holds for restrictions to $X_1$.
\end{rem}

\subsection{\texorpdfstring{$\bA¹$-$h$}{A\textonesuperior-h}-cobordism of projective bundles}

Given any two $\bA¹$-concordant vector bundles over a smooth projective variety
$X$, we will show in Lemma~\ref{lem:concob} that the associated projectivized
bundles are $\bA¹$-$h$-cobordant.  At the moment, this is the only source of
$\bA¹$-$h$-cobordisms between projective bundles at our disposal.  As an
immediate corollary, we obtain in Proposition~\ref{prop:cobordismclass} an
$\bA¹$-$h$-cobordism classification of those projective bundles that are
deformable to split bundles.

\begin{lem}[Construction of direct $\bA¹$-$h$-cobordisms from concordances]\label{lem:concob}
  Assume $X$ is a smooth projective abstract variety.  Let $\sE → X ⨯ \bA¹$ be a
  direct $\bA¹$-concordance between vector bundles $ι_0^* \sE$ and $ι_1^*\sE$.
  Then, projectivization induces a direct $\bA¹$-$h$-cobordism
  $$
  \xymatrix{ %
    \bP_{X⨯\bA¹}(\sE) \ar[r]^g & X⨯\bA¹ \ar[r] & \bA¹
  }
  $$
  between the projective bundles $\bP_{X}(ι_0^*\sE)$ and $\bP_{X}(ι_1^*\sE)$.
\end{lem}
\begin{proof}
  This is a special case of \cite[Prop.~3.1.5]{AM}, with $Y = \bP¹$, $Z =
  \bP_{X⨯\bA¹}(\sE)$, and $U = \bigsqcup U_i$ a suitable open affine cover of
  $X$.  To apply the result, we need to check condition (LT) in loc.cit.
  Starting with an arbitrary open, affine cover $V_i → X$ of $X$, the pullback
  of $g$ along $v_i⨯\Id : V_i⨯\bA¹ → X⨯\bA¹$ has the form $g_{v_i} :
  \bP_{V_i⨯\bA¹}(\sE|_{V_i⨯\bA¹}) → V_i⨯\bA¹$.  In other words, $g_{v_i}$ is the
  projectivization of a rank-two vector bundle over the affine scheme
  $V_i⨯\bA¹$.  By the homotopy invariance results of Quillen-Suslin
  \cite{quillen} and Lindel \cite{lindel}, any such vector bundle is the
  pullback of a rank-two vector bundle from $V_i$.  We can then further refine
  $V$ to a covering $U$ of $X$ such that the restriction of $\sE$ to $U⨯\bA¹$ is
  in fact the trivial rank-two bundle.  For the Zariski covering $U$ of $X$, the
  condition (LT) in \cite[Prop.~3.1.5]{AM} is satisfied, which proves the claim.
\end{proof}

\begin{prop}[$\bA¹$-$h$-cobordism classification of projective bundles deformable to split ones]\label{prop:cobordismclass}
  Fix integers $c_1$ and $c_2$, and assume that there exists $d ∈ \bN$ such that
  $d²-dc_1+c_2=0$.  Let $\sE_1$ and $\sE_2$ be two rank-two bundles on $\bP²$
  with Chern classes $c_1$ and $c_2$.  Then $\bP_{\bP²}(\sE_1)$ and
  $\bP_{\bP²}(\sE_2)$ are $\bA¹$-$h$-cobordant.
\end{prop}
\begin{proof}
  It suffices to show that the projectivization $\sE_1$ is $\bA¹$-$h$-cobordant
  to the projectivization of a split bundle.  Under the hypothesis on the Chern
  classes, the bundle $\sE_1(-d)$ has first Chern class $c_1-d$ and second Chern
  class $0$.  As a consequence, Theorem~\ref{thm:concordanceclass} shows that
  $\sE_1(-d)$ is $\bA¹$-concordant to the split bundle $\sO \oplus \sO(c_1-d)$.
  Tensoring the chain of $\bA¹$-concordances guaranteed by
  Theorem~\ref{thm:concordanceclass} with $\sO(d)$, one obtains a chain of
  $\bA¹$-concordances between $\sE_1$ and a split bundle.  Applying
  Lemma~\ref{lem:concob} to these $\bA¹$-concordances provides the required
  chain of $\bA¹$-$h$-cobordisms.
\end{proof}

\subsection{Direct \texorpdfstring{$\bA¹$-$h$}{A\textonesuperior-h}-cobordisms of projective bundles}

While Proposition~\ref{prop:cobordismclass} might suggest that there are large
classes of projectivized vector bundles that are all $\bA¹$-$h$-cobordant, the
following theorem asserts that few are in fact directly $\bA¹$-$h$-cobordant.

\begin{thm}[Non-existence of direct $\bA¹$-$h$-cobordisms]\label{thm:a1cobx}
  Fix integers $c_1, c_2 ∈ \bZ$.  There are infinitely many rank-two vector
  bundles $(\sE_j)_{j ∈ \bN}$ on $\bP²$ with Chern classes $c_1$ and $c_2$ such
  that no two of the projectivizations $\bP_{\bP²}(\sE_j)$ are directly
  $\bA¹$-$h$-cobordant.
\end{thm}

Before proving Theorem~\ref{thm:a1cobx} in Section~\ref{pf:thm:a1cobx} below, we
draw an immediate corollary and add a few comments.

\begin{cor}\label{cor:equiv}
  Direct $\bA¹$-$h$-cobordism fails to be an equivalence relation.
\end{cor}
\begin{proof}
  Fix integers $c_1$, $c_2$ and assume there exists an integer $d$ such that
  $d²-dc_1+c_2=0$.  By Proposition~\ref{prop:cobordismclass}, any two bundles
  with these Chern classes will be $\bA¹$-$h$-cobordant.  However, by
  Theorem~\ref{thm:a1cobx}, there is an infinite set of bundles no two of which
  are directly $\bA¹$-$h$-cobordant.
\end{proof}

\begin{rem}
  Note that $h$-cobordism of smooth manifolds {\em is} an equivalence relation:
  the obvious composition of two $h$-cobordisms is not a smooth manifold and not
  parametrized by the unit interval, but it can be smoothed and re-parametrized.
  The above shows that such a smoothing is not, in general, possible in
  algebraic geometry.
\end{rem}

\begin{question}
  ---
  \begin{enumerate}
  \item If $X$ is any smooth projective variety that is $\bA¹$-$h$-cobordant to
    a $\bP¹$-bundle over $\bP²$, does $X$ have the structure of a $\bP¹$-bundle
    over $\bP²$?  It seems likely to us that the answer is no: the examples
    coming to mind are non-trivial rank three vector bundles over $\bP¹$
    deformable to the trivial one.

  \item Do there exist varieties that have the $\bA¹$-homotopy type of a
    projective bundle but are not $\bA¹$-$h$-cobordant to a projective bundle?

  \item The techniques developed here do not provide an $\bA¹$-$h$-cobordism
    classification for projective space bundles that are not deformable to split
    bundles since we do not know the $\bA¹$-concordance classification of such
    bundles.  However, non-deformability results for vector bundles, would not
    imply non-existence of $\bA¹$-$h$-cobordisms: the $\bA¹$-$h$-cobordisms
    could go through singular fibers or simply fibers which have no projective
    bundle structure, and there are presently no methods to
    prove that a map with singular fibers is an $\bA¹$-$h$-cobordism.
  \end{enumerate}
\end{question}

\subsection*{Proof of Theorem~\ref*{thm:a1cobx}}
\label{pf:thm:a1cobx}

The proof of Theorem~\ref{thm:a1cobx} relies on Strømme's results concerning
deformations of vector bundles, as outlined in Section~\vref{ssec:deformation}.

\subsubsection*{Step 1: Simplification of Chern numbers}

Given any rank-two bundle $\sE$ on $\bP²$ and any invertible sheaf $\sL ∈
\Pic(\bP²)$, then $\bP_{\bP²}(\sE)$ and $\bP_{\bP²}(\sE \otimes \sL)$ are
canonically isomorphic.  To prove Theorem~\ref{thm:a1cobx}, it will therefore
suffice to consider the case where $c_1 ∈ \{0,-1\}$.

\subsubsection*{Step 2: Construction of bundles ${\sE_j}$}

Fix a number $c_2 ∈ \bN$, and recall from Proposition~\ref{prop:stroemmeX}
that there exists a (large) number $D \gg 3+c_1$, with the following property.
Given any number $j ∈ \bN$, there exists a rank-two vector bundle $\sE_j$ on
$\bP²$ with Chern classes $c_1$ and $c_2$ and with splitting type $d(j) := D+j$
that does not appear as a fiber in any family of bundles on $\bP²$ that is
generically of type less than $d(j)$.  Choose such $D$ and $(\sE_j)_{j ∈ \bN}$
and fix that choice for the remainder of the proof.

We aim to show that no two of the $\bP¹$-bundles $\bP_{\bP²}(\sE_j)$ are
directly $\bA¹$-$h$-cobordant.  We argue by contradiction and assume that the
following holds.

\begin{assumption}
  There exist two distinct numbers $a_0, a_1 ∈ \bN$ and a direct
  $\bA¹$-$h$-cobordism $f: \hcob{X} → \bA¹$ where $X_0 := f^{-1}(0) \cong
  \bP_{\bP²}(\sE_{a_0})$ and $X_1 := f^{-1}(1) \cong \bP_{\bP²}(\sE_{a_1})$.
\end{assumption}

\begin{rem}\label{rem:stue}
  Since $a_0 \ne a_1$, it follows from construction that $d(a_0) \ne d(a_1)$.
\end{rem}

\subsubsection*{Step 3: Extending the bundles to open neighborhoods}

If $i ∈ \{0,1\}$ is any given index, it follows from the deformation rigidity of
$\bP¹$-bundles over $\bP²$, Theorem~\vref{thm:rigidity}, that there exist open
neighborhoods $U_i = U(i) ⊆ \bA¹$, rank-two vector bundles $\sE_{U_i}$ over
$\bP² ⨯ U_i$ and commutative diagrams as follows,
$$
\xymatrix{ %
  \hcob{X}_{U_i} \ar[rrrr]^{f_{U_i}} \ar@{<->}[d]_{φ_{U_i}}^{\cong} &&&& U_i \ar@{<->}[d]^{=} \\
  \bP_{\bP² ⨯ U_i}(\sE_{U_i}) \ar[rr]_{\text{bundle map}}^{α_{U_i}} && \bP² ⨯ {U_i} \ar[rr]_{\text{projection}}^{β_{U_i}} && U_i \\
  \bP_{\bP²}(\sE_{a_i}) \ar[rr]_{\text{bundle map}}^{α_i} \ar@{^(->}[u] && \bP² \ar[rr]_{\text{constant}}^{β_i} \ar@{^(->}[u] && \{i\} \ar@{^(->}[u] \\
}
$$
where $\hcob{X}_{U_i} := f^{-1}(U_i)$ and $f_{U_i} := f|_{\hcob{X}_{U_i}}$.

Using the semicontinuity of splitting types, Fact~\vref{fact:semincontST}, and
using the assumption that the bundles $\sE_{a_i}$ do not appear as a fiber in
any family that is generically of type less than $d(a_i)$, we are free to shrink
the open sets $U_i$ and assume that the following holds in addition.

\begin{asswlog}\label{asswlog:stcf}
  The generic splitting type is constant in the families $\sE_{U_i}$.  More
  precisely, given any closed point $t ∈ U_i$, then the bundle $\sE_{U_i}|_{\bP²
    ⨯ \{t\}}$ has generic splitting type $d(a_i)$.
\end{asswlog}

\subsubsection*{Step 4: End of proof}

Choose any closed point $t ∈ U_0 ∩ U_1$.  We obtain identifications $\bP_{\bP²}
(\sE_{a_0}) \cong X_t \cong \bP_{\bP²} (\sE_{a_1})$.  Now, since the splitting
types $d(a_i)$ are larger than $3+c_1$, if follows from the uniqueness of the
bundle structure, Theorem~\vref{thm:uniq}, that there exists an automorphism $ψ
∈ \Aut(\bP²)$ and a commutative diagram,
$$
\xymatrix{%
  \bP_{\bP²} (\sE_{a_0}) \ar@{<->}[r]^(.6){\cong} \ar[d]_{\text{bundle map}} & X_t \ar@{<->}[r]^(.4){\cong} & \bP_{\bP²} (\sE_{a_1}) \ar[d]^{\text{bundle map}} \\
  \bP² \ar[rr]_{ψ}^{\cong} && \bP².
}
$$
It follows that the bundle $\sE_{a_0}$ and the pull-back $ψ^* \sE_{a_1}$ differ
only by the twist with a suitable line bundle, say $\sL ∈ \Pic(\bP²)$.
However, since the Chern classes of the two bundles agree, it follows that $\sL
\cong \sO_{\bP²}$, hence $\sE_{a_0} \cong ψ^* \sE_{a_1}$.  In particular, we
obtain that the generic splitting types of $\sE_{a_0}$ and $\sE_{a_1}$ agree.
By Assumption~\ref{asswlog:stcf}, this means that $d(a_0) = d(a_1)$, which
contradicts Remark~\ref{rem:stue}.  This finishes the proof of
Theorem~\ref{thm:a1cobx}.  \qed

%
%

\section{Complex realization}
\label{sec:07}

In the final section, we specialize to the case $k = \bC$ and compare the
algebraic classification results proven in this paper to their complex-geometric
counterparts.

\subsection{Comparison maps}
We have the following diagram of sets of equivalence classes of
projectivized rank-two bundles on $\bP²$, where the left column contains sets of
algebraic varieties modulo algebraic equivalence relations and the right column
contains complex manifolds modulo complex-geometric equivalence relations,
\begin{equation}\label{eq:CRD}
  \begin{tikzpicture}[baseline=(current bounding box.center)]
    \matrix[row sep=1cm,column sep=2cm] {
      \node (alghc) {$\begin{Bmatrix}
          \text{projective bundles } \bP_{\bP²(\bC)}(\sE)\\
          \text{modulo $\bA¹$-$h$-cobordism}
        \end{Bmatrix}$};
      &\node (DEF) {$\begin{Bmatrix}
          \textrm{complex manifolds } \bP_{\bC\bP²}(\sE)\\
          \textrm{modulo deformation equivalence}
        \end{Bmatrix}$}; \\
      &\node (DIFF) {$\begin{Bmatrix}
          \textrm{complex manifolds } \bP_{\bC\bP²}(\sE)\\
          \textrm{modulo diffeomorphism}
        \end{Bmatrix}$}; \\
      \node (algwe) {$\begin{Bmatrix}
          \text{projective bundles } \bP_{\bP²(\bC)}(\sE)\\
          \text{modulo $\bA¹$-weak equivalence}
        \end{Bmatrix}$};
      &\node (HTPY) {$\begin{Bmatrix}
          \textrm{complex manifolds } \bP_{\bC\bP²}(\sE)\\
          \textrm{modulo homotopy equivalence}
        \end{Bmatrix}$}; \\
    };
    \path (alghc) edge [->] node [above] {$φ_1$} node [below]
    {Prop.~\ref{prop:cobordismclass}} (DEF);
    \path (DEF) edge [->] node [right] {$φ_4$} node [left]
    {Props.~\ref{prop:hodef}, \ref{prop:hodiff}} (DIFF);
    \path (DIFF) edge [->] node [right] {$φ_5$} node [left]
    {Prop.~\ref{prop:hodiff}} (HTPY);
    \path (alghc) edge [->] node [left] {$φ_3$} node [right]
    {Prop.~\ref{prop:cobordismclass}} (algwe);
    \path (algwe) edge [->] node [above] {$φ_2$} node [below]
    {Prop.~\ref{prop:XX}} (HTPY);
  \end{tikzpicture}
\end{equation}

\begin{explanation}
  The two horizontal arrows $φ_1$ and $φ_2$ are both induced by complex
  realization $X \mapsto X(\bC)$, while the vertical arrows $φ_3$, $φ_4$ and
  $φ_5$ are induced by the identity map.  We briefly show that each of these
  maps is in fact well-defined.
  \begin{enumerate}
  \item If $f : X → \bA¹$ is any $\bA¹$-$h$-cobordism, then $f$ is a smooth
    morphism over a Zariski neighborhood $U$ of $\{0,1\} ⊆ \bA¹$, see
    Remark~\ref{rem:2}.  The complex realization $f(\bC) : X(\bC) → \bC$,
    restricted to $U(\bC)$ provides a deformation from $X_0(\bC)$ to $X_1(\bC)$.
    It follows that $φ_1$ is well-defined.

  \item Recall that the assignment that sends a smooth $k$-scheme $X$ to
    $X(\bC)$ equipped with its usual structure of a complex manifold extends to
    a ``complex realization'' functor $\mathfrak{R}_{ι}$ from the
    $\bA¹$-homotopy category $\ho{k}$ to the usual homotopy category of
    topological spaces, \cite[§~3.3]{MV}.  Using a slightly different model
    structure on $\Spc_k$, Dugger and Isaksen showed in
    \cite[Thm.~5.2]{DuggerIsaksen} that the complex realization functor between
    homotopy categories is actually part of a Quillen adjunction.  In
    particular, in their model structure, $\bA¹$-weak equivalences between
    smooth schemes are sent to weak equivalences of the associated topological
    spaces.  It follows that $φ_2$ is well-defined.

  \item If $f:X→\bA¹$ is an $\bA¹$-$h$-cobordism, then the inclusions
    $f^{-1}(0)→ X$ and $f^{-1}(1) → X$ are $\bA¹$-weak equivalences by
    definition.  In particular, $f^{-1}(0)$ and $f^{-1}(1)$ have the same weak
    $\bA¹$-homotopy type.  It follows that $φ_3$ is well-defined.

  \item Deformations as complex manifolds induce diffeomorphisms, by Ehresmann's
    fibration theorem.  This shows that $φ_4$ is well-defined.

  \item Diffeomorphism of complex manifolds are homotopy equivalences, hence
    $φ_5$ is well-defined.
  \end{enumerate}
\end{explanation}

\subsection{Homotopy classification}

Our results imply that the homotopy classification results in the above diagram
agree.  In other words, we show that the $\bA¹$-homotopy invariants do not
contain any more information than the classical algebraic-topological invariants
for $\bP¹$-bundles over $\bP²$.

\begin{prop}\label{prop:XX}
  The map $φ_2$ in Diagram~\eqref{eq:CRD} is a bijection.
\end{prop}
\begin{proof}
  Surjectivity of $φ_2$ follows by GAGA, since every projectivization of a
  holomorphic vector bundle over $\bC\bP²$ has an algebraic structure.
  Injectivity can be seen as follows.  Given two varieties
  $\bP_{\bC\bP²}(\sE_1)$, $\bP_{\bC\bP²}(\sE_2)$ whose complex realizations are
  homotopy equivalent, then their cohomology rings are isomorphic.  In the
  situation at hand, the cycle class maps are isomorphisms, so we have an
  isomorphism of Chow rings.  As in the proof of
  Theorem~\ref{thm:a1homotopyp1bundles}, the isomorphism of Chow rings implies
  that the Chern classes of $\sE_1$ and $\sE_2$ are in the same
  $\Pic(\bP²)$-orbit, and hence $\bP_{\bC\bP²}(\sE_1)$ and
  $\bP_{\bC\bP²}(\sE_2)$ are $\bA¹$-weakly equivalent.
\end{proof}

\begin{rem}
  Alternatively, one can use a proof as in Section~\ref{sec:04} together with
  Peterson's classification of complex rank $n$ vector bundles on projective
  spaces, \cite{peterson}, to see that projectivizations of holomorphic vector
  bundles over $\bC\bP²$ are classified up to homotopy equivalence by the exact
  same $\bZ$-orbits of $\bZ \cdot H \oplus \bZ \cdot H²$.
\end{rem}

\subsection{Complex-geometric classification}

The vertical maps $φ_4$ and $φ_5$ of Diagram~\eqref{eq:CRD} are also bijections.
In other words, all relevant equivalence relations agree on the set of
isomorphism classes of $\bC\bP¹$-bundles over $\bC\bP²$: cohomology ring
isomorphisms, homotopy equivalence, diffeomorphism and deformation equivalence.

\begin{prop}\label{prop:hodef}
  The composition $φ_5 ◦ φ_4$ the maps in Diagram~\eqref{eq:CRD} is a bijection.
\end{prop}
\begin{proof}
  Surjectivity is clear because $φ_5 ◦ φ_4$ is induced from the identity map.
  If two bundles $\bP_{\bC\bP²}(\sE_1)$, $\bP_{\bC\bP²}(\sE_2)$ are homotopy
  equivalent, then their Chern classes are in the same $\Pic(\bP²)$-orbit,
  cf.~Proposition~\ref{prop:XX} and its proof.  Tensoring by a line bundle, we
  can assume that the Chern classes of $\sE_1$ and $\sE_2$ agree.
  Proposition~\ref{prop:explicit1-x} asserts that any two rank-two vector
  bundles over $\bP²$ whose Chern classes agree are equivalent by the
  equivalence relation generated by deformations over an irreducible base.  This
  implies the claim.
\end{proof}

\begin{prop}\label{prop:hodiff}
  The map $φ_5$ in Diagram~\eqref{eq:CRD} is a bijection.
\end{prop}
\begin{proof}
  The easiest way to see this is to note that the map from projective bundles
  modulo deformation equivalence to projective bundles modulo diffeomorphism is
  surjective.

  We also outline how the bijectivity can be deduced from the work of Okonek and
  van de Ven \cite{okonek:vandeven}.  First note that from the cohomology ring
  computation in Computation~\ref{comp:chow} we see that the integral homology
  of $\bP_{\bC\bP²}(\sE)$ is torsion free, and the only non-trivial Betti
  numbers are $b_2=b_4=2$ and $b_0=b_6=1$.  In the case of a complex manifold,
  $w_2$ can be obtained as the mod $2$ reduction of $c_1$, and the second Chern
  class determines $p_1$ via $c_2=\frac{1}{2}(c_1²-p_1)$, cf.\
  \cite[Prop.~8]{okonek:vandeven}.  The computations of
  \cite[Prop.~15]{okonek:vandeven} show that the diffeomorphism invariants of
  the projective bundle are determined completely by the Chern classes of the
  projective bundle.  Finally, \cite[Prop.~17]{okonek:vandeven} (or rather its
  proof) shows that all invariants are realizable by projective bundles.
  Summing up, the results of Okonek and van de Ven cited above show that
  diffeomorphism classes of $\bC\bP¹$-bundles over $\bC\bP²$ are described by
  the same invariants as the homotopy equivalence classes of such bundles.  This
  proves the claim.
\end{proof}

\begin{rem}
  The generic splitting type of an unstable vector bundle can also be seen from
  the corresponding projective bundle.  Given a $\bP¹$-bundle over $\bP²$, its
  restriction to a projective line $\ell ⊆ \bP²$ is a Hirzebruch surface
  $\bF_a$.  Generically, it is the Hirzebruch surface $\bF_{c_1-2d}$ if $d$ is
  the generic splitting type.  There are some lines where we get a different
  Hirzebruch surface $\bF_{e}$, with $e\equiv c_1-2d\mod 2$.  These lines, when
  viewed as points of the dual projective plane, form a curve, the \emph{curve
    of jumping lines}.  The above description provides one explanation why no
  difference between such projective bundles is visible in $\bA¹$-homotopy or
  the diffeomorphism type - the differences between the Hirzebruch surfaces
  $\bF_d$ and $\bF_{d'}$ for $d\equiv d'\mod 2$ are not visible in either
  setting.
\end{rem}

\subsection{Relation to \texorpdfstring{$\bA¹$-$h$}{A\textonesuperior-h}-cobordism classification}

In the case of bundles that are deformable to split bundles, that is, in the
case where there exists an integer $d$ with $d²-dc_1+c_2=0$,
Proposition~\ref{prop:cobordismclass} shows that the two remaining arrows in the
diagram are bijections as well.  The $\bA¹$-$h$-cobordism classification agrees
with the $\bA¹$-homotopy classification as well as all the complex-geometric
equivalence relations.  It is, however, not clear (to us) what happens for
bundles that are not deformable to split bundles.



\begin{thebibliography}{OVdV95}

\bibitem[AF14a]{AsokFaselThreefolds}
A.~Asok and J.~Fasel.
\newblock A cohomological classification of vector bundles on smooth affine
  threefolds.
\newblock {\em To appear in Duke Math. J.}, 2014.
\newblock Preprint \href{http://arxiv.org/abs/1204.0770}{arXiv:1204.0770}.

\bibitem[AF14b]{AsokFaselA3minus0}
A.~Asok and J.~Fasel.
\newblock Splitting vector bundles outside the stable range and homotopy theory
  of punctured affine spaces.
\newblock {\em To appear in J. Amer. Math. Soc.}, 2014.
\newblock Preprint \href{http://arxiv.org/abs/1209.5631}{arXiv:1209.5631}.

\bibitem[AM11]{AM}
A.~Asok and F.~Morel.
\newblock Smooth varieties up to {${\mathbb A}^1$}-homotopy and algebraic
  {$h$}-cobordisms.
\newblock {\em Adv. Math.}, 227(5):1990--2058, 2011.

\bibitem[Arr07]{Arrondo}
E.~Arrondo.
\newblock A home-made {H}artshorne-{S}erre correspondence.
\newblock {\em Rev. Mat. Complut.}, 20(2):423--443, 2007.

\bibitem[Aso13]{AsokPi1}
A.~Asok.
\newblock Splitting vector bundles and {${\mathbb A}^1$}-fundamental groups of
  higher dimensional varieties.
\newblock {\em J. Top.}, 2013.
\newblock
  \href{http://dx.doi.org/10.1112/jtopol/jts034}{DOI:10.1112/jtopol/jts034}.

\bibitem[B{\u{a}}n83]{banica}
C.~B{\u{a}}nic{\u{a}}.
\newblock Topologisch triviale holomorphe {V}ektorb\"undel auf {${\bf
  P}^{n}({\bf C})$}.
\newblock {\em J. Reine Angew. Math.}, 344:102--119, 1983.

\bibitem[Bor91]{B91}
A.~Borel.
\newblock {\em Linear Algebraic Groups}, volume 126 of {\em Graduate Texts in
  Mathematics}.
\newblock Springer, 1991.

\bibitem[DI04]{DuggerIsaksen}
D.~Dugger and D.~C. Isaksen.
\newblock Topological hypercovers and {$\mathbb{A}^1$}-realizations.
\newblock {\em Math. Z.}, 246(4):667--689, 2004.

\bibitem[DW82]{DW82}
R.~Dedekind and H.~Weber.
\newblock Theorie der algebraischen {F}unktionen einer {V}eränderlichen.
\newblock {\em Crelle J. Reine Angew. Math}, 92:181--290, 1882.

\bibitem[Fuj87]{MR946238}
T.~Fujita.
\newblock On polarized manifolds whose adjoint bundles are not semipositive.
\newblock In {\em Algebraic geometry, {S}endai, 1985}, volume~10 of {\em Adv.
  Stud. Pure Math.}, pages 167--178. North-Holland, Amsterdam, 1987.

\bibitem[God73]{MR0345092}
R.~Godement.
\newblock {\em Topologie alg\'ebrique et th\'eorie des faisceaux}.
\newblock Hermann, Paris, 1973.
\newblock Troisi{\`e}me {\'e}dition revue et corrig{\'e}e, Publications de
  l'Institut de Math{\'e}matique de l'Universit{\'e} de Strasbourg, XIII,
  Actualit{\'e}s Scientifiques et Industrielles, No. 1252.

\bibitem[Gro58]{Gro5658}
A.~Grothendieck.
\newblock Compléments de géométrie algébrique. {E}spaces de transformations
  (exposé no.~5.).
\newblock In {\em Classification des groupes de Lie algébriques}, number~1 in
  Séminaire Claude Chevalley. Paris, 1956--58.
\newblock Freely available on the Numdam web site at
  \url{http://www.numdam.org/item?id=SCC_1956-1958__1__A5_0}.

\bibitem[Gro61]{EGA2}
A.~Grothendieck.
\newblock {\'E}l{\'e}ments de g{\'e}om{\'e}trie alg{\'e}brique. {II}.
  {{\'E}}tude globale {\'e}l{\'e}mentaire de quelques classes de morphismes.
\newblock {\em Inst. Hautes {\'E}tudes Sci. Publ. Math.}, (8):222, 1961.
\newblock Revised in collaboration with Jean Dieudonn{\'e}. Freely available on
  the Numdam web site at
  \url{http://www.numdam.org/numdam-bin/feuilleter?id=PMIHES_1961__8_}.

\bibitem[Gro66]{EGA4-3}
A.~Grothendieck.
\newblock {\'E}l{\'e}ments de g{\'e}om{\'e}trie alg{\'e}brique. {IV}.
  {{\'E}}tude locale des sch{\'e}mas et des morphismes de sch{\'e}mas.
  {T}roisi{\`e}me partie.
\newblock {\em Inst. Hautes {\'E}tudes Sci. Publ. Math.}, (28):255, 1966.
\newblock Revised in collaboration with Jean Dieudonn{\'e}. Freely available on
  the Numdam web site at
  \url{http://www.numdam.org/numdam-bin/feuilleter?id=PMIHES_1966__28_}.

\bibitem[Gro71]{SGA1}
A.~Grothendieck.
\newblock {\em Rev{\^e}tements {\'e}tales et groupe fondamental (SGA 1)}.
\newblock Springer-Verlag, Berlin, 1971.
\newblock S{\'e}minaire de G{\'e}om{\'e}trie Alg{\'e}brique du Bois Marie
  1960--1961, Dirig{\'e} par Alexandre Grothendieck. Augment{\'e} de deux
  expos{\'e}s de M. Raynaud, Lecture Notes in Mathematics, Vol. 224. Also
  available as \href{http://arxiv.org/abs/math/0206203}{arXiv:math/0206203}.

\bibitem[Har77]{Ha77}
R.~Hartshorne.
\newblock {\em Algebraic geometry}.
\newblock Springer-Verlag, New York, 1977.
\newblock Graduate Texts in Mathematics, No. 52.

\bibitem[Har78]{HartshorneSB}
R.~Hartshorne.
\newblock Stable vector bundles of rank {$2$} on {${\bf P}^{3}$}.
\newblock {\em Math. Ann.}, 238(3):229--280, 1978.

\bibitem[Kle69]{Kleiman}
S.~L. Kleiman.
\newblock Geometry on {G}rassmannians and applications to splitting bundles and
  smoothing cycles.
\newblock {\em Inst. Hautes \'Etudes Sci. Publ. Math.}, (36):281--297, 1969.

\bibitem[Kol96]{K96}
J.~Koll{á}r.
\newblock {\em Rational curves on algebraic varieties}, volume~32 of {\em
  Ergebnisse der Mathematik und ihrer Grenzgebiete. 3. Folge. A Series of
  Modern Surveys in Mathematics}.
\newblock Springer-Verlag, Berlin, 1996.

\bibitem[Lin81]{lindel}
H.~Lindel.
\newblock On the {B}ass-{Q}uillen conjecture concerning projective modules over
  polynomial rings.
\newblock {\em Invent. Math.}, 65:319--323, 1981.

\bibitem[McC01]{MR1793722}
J.~McCleary.
\newblock {\em A user's guide to spectral sequences}, volume~58 of {\em
  Cambridge Studies in Advanced Mathematics}.
\newblock Cambridge University Press, Cambridge, second edition, 2001.

\bibitem[Mor12]{MField}
F.~Morel.
\newblock {\em {${\mathbb A}^1$}-algebraic topology over a field}, volume 2052
  of {\em Lecture Notes in Mathematics}.
\newblock Springer, Heidelberg, 2012.

\bibitem[Mum08]{MR2514037}
D.~Mumford.
\newblock {\em Abelian varieties}, volume~5 of {\em Tata Institute of
  Fundamental Research Studies in Mathematics}.
\newblock Published for the Tata Institute of Fundamental Research, Bombay,
  2008.
\newblock With appendices by C. P. Ramanujam and Yuri Manin, Corrected reprint
  of the second (1974) edition.

\bibitem[MV99]{MV}
F.~Morel and V.~Voevodsky.
\newblock {${\mathbb A}^1$}-homotopy theory of schemes.
\newblock {\em Inst. Hautes \'Etudes Sci. Publ. Math.}, 90:45--143 (2001),
  1999.

\bibitem[OVdV95]{okonek:vandeven}
Ch. Okonek and A.~Van~de Ven.
\newblock Cubic forms and complex {$3$}-folds.
\newblock {\em Enseign. Math. (2)}, 41(3-4):297--333, 1995.

\bibitem[Pet59]{peterson}
F.P. Peterson.
\newblock Some remarks on {C}hern classes.
\newblock {\em Ann. of Math. (2)}, 69:414--420, 1959.

\bibitem[Qui76]{quillen}
D.G. Quillen.
\newblock Projective modules over polynomial rings.
\newblock {\em Invent. Math}, 36:167--171, 1976.

\bibitem[Sch61]{schwarzenberger:surfaces}
R.L.E. Schwarzenberger.
\newblock Vector bundles on algebraic surfaces.
\newblock {\em Proc. London Math. Soc. (3)}, 11:601--622, 1961.

\bibitem[Str83]{stroemme}
S.A. Str{\o}mme.
\newblock Deforming vector bundles on the projective plane.
\newblock {\em Math. Ann.}, 263:385--397, 1983.

\bibitem[Vak06]{Vakil}
R.~Vakil.
\newblock Murphy's law in algebraic geometry: badly-behaved deformation spaces.
\newblock {\em Invent. Math.}, 164(3):569--590, 2006.

\bibitem[Wei89]{WeibelHAK}
C.A. Weibel.
\newblock Homotopy algebraic {$K$}-theory.
\newblock In {\em Algebraic {$K$}-theory and algebraic number theory
  ({H}onolulu, {HI}, 1987)}, volume~83 of {\em Contemp. Math.}, pages 461--488.
  Amer. Math. Soc., Providence, RI, 1989.

\bibitem[Wen10]{chev-rep}
M.~Wendt.
\newblock {$\mathbb{A}^1$}-homotopy of {C}hevalley groups.
\newblock {\em J. K-Theory}, 5(2):245--287, 2010.

\bibitem[Wen11]{WendtTorsor}
M.~Wendt.
\newblock Rationally trivial torsors in {$\mathbb{A}^1$}-homotopy theory.
\newblock {\em J. K-Theory}, 7(3):541--572, 2011.

\end{thebibliography}
\end{document}